\documentclass[a4paper,12pt]{article}
\usepackage{amsmath}
\usepackage{amsfonts}
\usepackage[mathscr]{eucal}
\usepackage{amssymb}
\usepackage{latexsym}
\usepackage{amsthm}
\usepackage{amscd}
\usepackage[dvips]{graphicx}
\theoremstyle{plain}
\newtheorem{theorem}{Theorem}[section]
\newtheorem{corollary}[theorem]{Corollary}
\newtheorem{lemma}[theorem]{Lemma}
\newtheorem{proposition}[theorem]{Proposition}

\theoremstyle{definition}
\newtheorem{definition}[theorem]{Definition}
\theoremstyle{remark}
\newtheorem{remark}{Remark}[section]

\newcommand{\knabla}{%
  \setbox0=\hbox{$\triangle $}%
  \raisebox{\ht0}{\rotatebox{180}{$\triangle $}}}

\begin{document}

\title{Regular flat structure and generalized Okubo system}
\author{Hiroshi Kawakami and Toshiyuki Mano}
\maketitle
\begin{abstract}

 We study a relationship between regular flat structures and generalized Okubo systems.
 We show that the space of variables of isomonodromic deformations of a regular generalized Okubo system 
 can be equipped with a flat structure.
 As its consequence, we introduce flat structures 
 on the spaces of independent variables of generic solutions to (classical) Painlev\'e equations (except for PI).
 In our framework, the Painlev\'e equations PVI-PII can be treated
  uniformly as just one system of differential equations called the four-dimensional extended WDVV equation.
 Then the well-known coalescence cascade of the Painlev\'e equations
 corresponds to the degeneration scheme of the Jordan normal forms of a square matrix of rank four.

 \vspace{5pt} 
\end{abstract}

\section{Introduction} \label{sec:introduction}

At the end of 1970's, K. Saito introduced the notion of flat structure 
in order to study the structure of the parameter spaces of isolated hypersurface singularities.
In particular he defined and constructed flat coordinates on the orbit spaces of finite real reflection groups in \cite{Sai2,SYS}
with his collaborators T. Yano and J. Sekiguchi.
After their pioneering work, B. Dubrovin \cite{Du} unified both the flat structure formulated by K. Saito and the WDVV equation
which arises from the 2D topological field theory as Frobenius manifold structure.
In his theory on Frobenius manifolds, Dubrovin studied a relationship between Frobenius manifolds 
and isomonodromic deformations of linear differential equations of certain type.
Especially he showed that there is a correspondence between three dimensional regular semisimple Frobenius manifolds
(massive Frobenius manifolds) and generic solutions to a one-parameter family of the sixth Painlev\'e equation.
Up to now, there are several generalizations of Frobenius manifold such as $F$-manifold by C. Hertling and Y. Manin \cite{HM,He}
and Saito structure (without metric) by C. Sabbah \cite{Sab}.
A. Arsie and P. Lorenzoni \cite{AL0,Lo} showed that three-dimensional regular semisimple bi-flat $F$-manifolds are parameterized 
by solutions to (full-parameter) Painlev\'e VI equation,
which may be regarded as an extension of Dubrovin's result. 
Furthermore Arsie-Lorenzoni \cite{AL1} showed that three-dimensional regular non-semisimple bi-flat $F$-manifolds are parameterized 
by solutions to Painlev\'e V and IV equations.

The theory of linear differential equations on a complex domain is a classical branch of Mathematics.
In recent years there has been great progress in this branch  as a turning point 
with the introduction of the notions of {\it middle convolution} and {\it rigidity index} by N. M. Katz \cite{Katz}.
Especially T. Oshima developed a classification theory of Fuchsian differential equations 
in terms of their rigidity indices and spectral types \cite{Oshi1,Oshi2,OshiB}. 
(\cite{HaBook} is a nice introductory text on the ``Katz-Oshima theory''.)
{\it Okubo systems} play a central role in these developments (cf. \cite{DR1,DR2,Oshi2,Yo}):
a matrix system of linear differential equations with the form
 \begin{equation} \label{eq:introOkubo}
   (z-T)\frac{dY}{dz}=-B_{\infty}Y,
 \end{equation}
 where $T, B_{\infty}$ are constant square matrices,
 is an Okubo system if $T$ is diagonalizable.
 If $T$ is not necessarily diagonalizable, (\ref{eq:introOkubo}) is said to be a {\it generalized Okubo system},
 which was studied by H. Kawakami \cite{Kaw,KawDT} in order to generalize the middle convolution 
 to linear differential equations with irregular singularities.

In this paper, we study isomonodromic deformations of a generalized Okubo system.
We show that the space of variables of isomonodromic deformations of a generalized Okubo system can be equipped with
a Saito structure (without a metric).
(In the sequel, we abbreviate ``Saito structure (without metric)'' to ``Saito structure'' for brevity.
In addition, ``flat structure'' in this paper stands for the same meaning as ``Saito structure'', c.f. \cite[\S 1]{KMS}.)
A major motivation of the present paper is Arsie-Lorenzoni's work \cite{AL1,Lo}, 
in which it was shown that three-dimensional regular bi-flat $F$-manifolds are parametrized by solutions to
the Painlev\'e IV,V,VI equations.
(It is proved in \cite{AL2,KoMiSh} that Arsie-Lorenzoni's bi-flat $F$-manifold is equivalent to Sabbah's Saito structure.)
 In \cite{AL1}, it has been left as an open problem how the remaining Painlev\'e equations (i.e. PIII, PII, PI) can be related with Saito structures.
 We give an answer to this problem in the present paper:
 we show that solutions to the Painlev\'e III and II are related with four-dimensional regular Saito structures.
As for the Painlev\'e I, it can be related with isomonodromic deformations of a generalized Okubo system of rank seven (which is minimal rank),
however it does not satisfy the regularity condition.
For this reason, PI can not be treated in the framework of this paper.
As another application, we study the initial value problem
for regular Saito structures.
L. David and C. Hertling \cite{DH} studied the initial condition theorem for regular $F$-manifolds and Frobenius manifolds.
$F$-manifold is a lower structure of Saito structure.
The initial condition for a regular Saito structure consists of that for the underlying regular $F$-manifold 
and a data to specify a flat coordinate system and its weight 
(c.f. Corollary~\ref{cor:initial} and Remark~\ref{rem:initialFS}).

This paper is constructed as follows.
In Section~\ref{sec:okuboiso}, we introduce the notion of {\it extended generalized Okubo system}
as a completely integrable Pfaffian system extending a generalized Okubo system.
We see that an extended generalized Okubo system is equivalent to
 an isomonodromic deformation of a generalized Okubo system (Lemma~\ref{lem:isomono}).
We observe that an extended generalized Okubo system naturally induces a {\it Saito bundle} on a complex domain,
which will be used in Section~\ref{subsec:Okubo}.
In Section~\ref{subsec:Saito}, we review the general theory of Saito structure.
Particularly we introduce a potential vector field and the extended WDVV equation for a Saito structure.
In Section~\ref{subsec:Okubo},
we show that the space of variables of an extended generalized Okubo system can be equipped with
a Saito structure under some generic condition (Theorem~\ref{saitojacobian}).
The arguments there closely follow \cite[Chapter VII]{Sab}:
we find a condition for that the Saito bundle induced by an extended generalized Okubo system in Section~\ref{sec:okuboiso}
has a {\it primitive section}.
In Section~\ref{sec:flatPainleve},
we introduce Saito structures on the spaces of independent variables of solutions to the Painlev\'e equations.
It is well known that each Painlev\'e equation (PI-PVI) can be derived 
from isomonodromic deformations of certain $2\times 2$ system of linear differential equations.
Kawakami \cite{KawDT,Kaw} showed that any system of linear differential equations can be transformed 
into a generalized Okubo system.
Using Kawakami's construction, we obtain a generalized Okubo system 
whose isomonodromic deformations are governed by solutions to each Painlev\'e equation.
Then in virtue of Theorem~\ref{saitojacobian},
we see that generic solutions to the Painlev\'e VI,V,IV equations correspond to
solutions to the three-dimensional extended WDVV equation 
satisfying some regularity condition (Theorem~\ref{thm:extWDVVandPVIPVPIV}),
which provides another proof of Arsie-Lorenzoni's result mentioned above.
Moreover generic solutions to the Painlev\'e III and II equations correspond to solutions 
to the four-dimensional extended WDVV equation.
In Section~\ref{sec:coalcas},
we treat uniformly the Painlev\'e VI-II equations.
In the previous section, we have seen that PIII and PII are related with four-dimensional regular Saito structures.
The system of linear differential equations whose isomonodromic deformations are governed by 
PVI,PV,PIV can also be transformed into generalized Okubo systems of rank four.
Then we see that generic solutions to PVI-PII correspond to solutions to the four-dimensional extended WDVV equation
 (Theorem~\ref{thm:WDVVPVI-PII}).
The well-known coalescence cascade of the Painlev\'e equations
\begin{equation} 
 \begin{CD}
   \mbox{PVI} @>>> \mbox{PV} @>>> \mbox{PIV}   @.  \\
   @.                             @VVV                     @VVV              @.  \\
   @.                            \mbox{PIII} @>>> \mbox{PII} @>>> \mbox{PI} 
 \end{CD}
\end{equation}
corresponds in our framework (except for PI) to the degeneration scheme of the Jordan normal forms of a square matrix of rank four:
\[
    \begin{CD}
   \begin{pmatrix}  z_{1} & {} & {} & O\\ {} & z_{2} & {} & {} \\ {} & {} & z_{3} & {} \\ O & {} & {} & z_{4} \end{pmatrix}
   @>>> \begin{pmatrix} z_{1} & 1 & {} & O\\ {} & z_{1} & {} & {} \\ {} & {} & z_{2} & {} \\ O & {} & {} & z_{3} \end{pmatrix}
   @>>> \begin{pmatrix} z_{1} & 1 & {} & O\\ {} & z_{1} & 1 & {} \\ {} & {} & z_{1} & {} \\ O & {} & {} & z_{2} \end{pmatrix} \\
   @.                             @VVV                     @VVV                \\
   @.       \begin{pmatrix} z_{1} & 1 & {} & O\\ {} & z_{1} & {} & {} \\ {} & {} & z_{2} & 1 \\ O & {} & {} & z_{2} \end{pmatrix}
   @>>> \begin{pmatrix} z_{1} & 1 & {} & O \\ {} & z_{1} & 1 & {} \\ {} & {} & z_{1} & 1 \\ O & {} & {} & z_{1}
   \end{pmatrix}.
 \end{CD}
\]
It is a rather unexpected aspect of the Painlev\'e equations that one can treat them (except PI) all together
 as just one system of differential equations.
It seems to be an interesting problem to study analytic properties of Painlev\'e transcendents uniformly
in terms of potential vector fields (see \cite{KMS1,KMS,KMS3,KMS4} for algebraic solutions to PVI).
We will study this problem elsewhere.
In Appendix \ref{app:isomonodormicdef}, we give a proof of Lemma~\ref{lem:isomono}.
For this aim, we briefly review the theory of isomonodromic deformations of systems of linear differential equations 
with irregular singularities following Jimbo-Miwa-Ueno \cite{JMU}.
In Appendix \ref{app:compOkubo}, following \cite{KawDT,Kaw},
we explain how to construct a generalized Okubo system from 
a given system of linear differential equations.
This construction is used in Sections~\ref{sec:flatPainleve} and~\ref{sec:coalcas}.
\vspace{5pt}
\\
{\bf Acknowledgments}
The authors thank Professor C. Hertling for reading a preprint and giving useful comments on it
and Professor M. Noumi for useful discussions.
This work was partially supported by JSPS KAKENHI Grant Numbers 25800082, 17K05335.

\section{Generalized Okubo system} \label{sec:okuboiso}

For $N\in\mathbb{N}$, let $T$ and $B_{\infty}$ be $N\times N$-matrices.
The system of ordinary linear differential equations
 \begin{equation} \label{eq:okubo}
   (zI_N-T)\frac{dY}{dz}=-B_{\infty}Y,
 \end{equation}
 where $I_N$ denotes the identity matrix of rank $N$,
 is said to be an {\it Okubo system} if the matrix $T$ is diagonalizable (cf. \cite{Ok}),
  and said to be a {\it generalized Okubo system} if $T$ is not necessarily diagonalizable (cf. \cite{Kaw}).
 Let $U$ be a simply-connected domain in $\mathbb{C}^N$.
We extend (\ref{eq:okubo}) to a completely integrable Pfaffian system of the form
\begin{equation} \label{eq:okubopfaff}
dY=-(zI_N-T)^{-1}\bigl(dz+\tilde{\Omega}\bigr)B_{\infty}Y
\end{equation}
on $\mathbb{P}^1\times U$, 
or equivalently an integrable meromorphic connection on the trivial bundle of rank $N$ over $\mathbb{P}^1\times U$:
\begin{equation}
  \nabla^{\mathcal{G}\mathcal{O}}:=d+(zI_N-T)^{-1}(dz+\tilde{\Omega})B_{\infty},
\end{equation}
where $z$ is a non-homogeneous coordinate of $\mathbb{P}^1$ and
 $\tilde{\Omega}$ is an $N\times N$ matrix valued $1$-form on $U$ satisfying the conditions below.
We put the following assumptions on (\ref{eq:okubopfaff}):

\begin{enumerate}

 \item[(A1)] $T$ (resp. $\tilde{\Omega}$) is an $N\times N$ matrix whose entries are holomorphic functions 
 (resp. holomorphic $1$-forms) on $U$.
 
 \item[(A2)] $B_{\infty}=\mbox{diag}(\lambda_1,\dots,\lambda_N)$,
where $\lambda_i\in\mathbb{C}$ satisfy $\lambda_i-\lambda_j\not\in \mathbb{Z}\setminus\{0\}$ for $i\not=j$.

 \item[(A3)] The matrix $T$ is generically regular. Namely, let $T$ have the Jordan normal form
    \begin{equation*} 
     T\sim J_1\oplus \cdots \oplus J_n,
   \end{equation*}
    where $J_k$ is a Jordan block of rank $m_k$:
   \begin{equation} \label{eq:Jnf}
      J_k=\begin{pmatrix}
          z_{k,0} & 1 & {} & O \\
          {} & \ddots & \ddots&{}  \\
          {} & {} & \ddots &1 \\
          O & {} & {} & z_{k,0}
         \end{pmatrix},
   \end{equation}
   and $m_k\in\mathbb{N}$, $1\leq k\leq n$ are subject to $m_1+\cdots +m_n=N$,
   and put
   \begin{gather*}
      H(z)=\det (zI_N-T)=\prod_{k=1}^n(z-z_{k,0})^{m_k}, \\
      H_{{\rm red}}(z)=\prod_{k=1}^{n}(z-z_{k,0}), \ \ 
      \delta_{H_{{\rm red}}}=\prod_{k<l}(z_{k,0}-z_{l,0})^2.
   \end{gather*}
   Then $\delta_{H_{{\rm red}}}$ is not identically zero on $U$.
\end{enumerate}

 By the assumption (A3), 
 there exists an invertible matrix $P$ whose entries are holomorphic functions 
 on an open set $W\subset U\setminus \{\delta_{H_{{\rm red}}}=0\}$
 such that
 \begin{equation} \label{eq:TandZ}
   P^{-1}TP=Z_1\oplus \cdots \oplus Z_n,
 \end{equation}
 where
 \begin{equation} \label{eq:Zk}
   Z_k=\begin{pmatrix}
              z_{k,0} & z_{k,1} & \cdots & z_{k,m_k-1} \\
                       {} & \ddots   & \ddots & \vdots \\
                       {} &        {}   & \ddots & z_{k,1} \\
                       O &        {}   &       {}  & z_{k,0}
           \end{pmatrix}, \ \ \ k=1,\dots ,n.
 \end{equation}
 Then we assume also the following (A4):
 \begin{enumerate}
   \item[(A4)]
   $$\bigwedge_{k=1}^ndz_{k,0}\wedge\cdots\wedge dz_{k,m_k-1}\neq 0$$ at generic points on $W$.
   This means that $(z_{k,0},\dots,z_{k,m_k-1})_{1\leq k\leq n}$ can be taken as a coordinate system 
   on a sufficiently small open set in $W$.
 \end{enumerate}
 Note that $Z_k$ is written as
  \[
   Z_k=\sum_{l=0}^{m_k-1}z_{k,l}\Lambda_k^l
 \]
 by introducing an $m_k\times m_k$-matrix $\Lambda_k$:
 \begin{equation} \label{eq:matrixLambda}
   \Lambda_k=\begin{pmatrix} 0 & 1 & {} & O \\ {} & \ddots & \ddots & {} \\ {} & {} & \ddots & 1 \\ O & {} & {} & 0 \end{pmatrix}.
 \end{equation}

\begin{lemma}\label{lem:lemma2.2}
 Assume that $B_{\infty}$ is invertible.
 Then the Pfaffian system $(\ref{eq:okubopfaff})$ is completely integrable
 if and only if the matrices $T, \tilde{\Omega}, B_{\infty}$ satisfy the following system of equations
 \begin{gather}
  [T,{\tilde \Omega}]=O,\quad {\tilde \Omega}\wedge {\tilde \Omega}=O, \label{eq:commute} \\
  dT+\tilde{\Omega}+[\tilde{\Omega},B_{\infty}]=O, \quad d\tilde{\Omega}=O. \label{eq:ci1} 
\end{gather}
 In particular, $T$ and $\tilde{\Omega}$ are simultaneously block diagonalizable:
 for an invertible matrix $P$, it holds that
\begin{equation} \label{eq:taikakuBi}
 P^{-1}TP=Z_1\oplus \cdots \oplus Z_n,\ \ \ \ 
 P^{-1}\tilde{\Omega}P=-(dZ_1\oplus \cdots \oplus dZ_n).
\end{equation}
\end{lemma}

\begin{proof}
  The assertions in the lemma were proved in \cite{KMS} in the case where $T$ in (\ref{eq:okubopfaff}) is diagonalizable.
   We prove the lemma in the general case by constructing a confluence process from the case where $T$ is diagonalizable.
   
   For $1\leq k\leq n$, we define an $m_k\times m_k$-matrix $P_k(\varepsilon)$ by 
   \[
      P_k(\varepsilon):=\sum_{l=0}^{m_k-1}a_{k,l}\Lambda_k^l,
   \]
   where $a_{k,l}=a_{k,l}(\varepsilon)$ is defined recursively by the following equalities:
   \[
      a_{k,0}=1,\ \ \ \ a_{k,l}=\frac{\sum_{j=0}^{l-1}a_{k,j}z_{k,l-j}}{l\varepsilon z_{k,1}},\ \ \ l=1,\dots ,m_k-1.
   \]
   We put
   \[
      Z_k(\varepsilon):=\mbox{diag}(z_{k,0},z_{k,0}+z_{k,1}\varepsilon, z_{k,0}+2z_{k,1}\varepsilon, \dots,z_{k,0}+(m_k-1)z_{k,1}\varepsilon).
   \]
   Let
   \[
      P_{\varepsilon}:=P_1(\varepsilon)\oplus \cdots \oplus P_n(\varepsilon),\ \ 
      Z_{\varepsilon}:=Z_1(\varepsilon)\oplus \cdots \oplus Z_n(\varepsilon).
   \]
   Then we readily see 
   \begin{equation} \label{eq:Zlimit}
     P_{\varepsilon}Z_{\varepsilon}P_{\varepsilon}^{-1} \to Z_1\oplus \cdots\oplus Z_n \ \ \mbox{ as } \varepsilon \to 0,
   \end{equation}
   from which it follows that
   \begin{equation}
     T_{\varepsilon}:=P(P_{\varepsilon}Z_{\varepsilon}P_{\varepsilon}^{-1})P^{-1}\to T \ \ \mbox{ as } \varepsilon \to 0,
   \end{equation}
   where $P$ is the same as in (\ref{eq:TandZ}).
  
   Since $Z_{\varepsilon}$ is diagonal, it is known in \cite{KMS} that the Pfaffian system
   \begin{equation} \label{eq:gouryuPfaff}
      dY=-(zI_N-T_{\varepsilon})^{-1}\bigl(dz+\tilde{\Omega}_{\varepsilon}\bigr)B_{\infty}Y
   \end{equation}
   is completely integrable if and only if
   $T_{\varepsilon}, \tilde{\Omega}_{\varepsilon}, B_{\infty}$ satisfy (\ref{eq:commute}),(\ref{eq:ci1})
   and $\tilde{\Omega}_{\varepsilon}$ is represented as
   \begin{equation} \label{eq:B(varepsilon)}
      \tilde{\Omega}_{\varepsilon}=-P(P_{\varepsilon}\,dZ_{\varepsilon}\,P_{\varepsilon}^{-1})P^{-1}.
   \end{equation}
   From the limit of (\ref{eq:B(varepsilon)}) and (\ref{eq:Zlimit}) as $\varepsilon\to 0$, we have
    \[
      P^{-1}\tilde{\Omega}P=-(dZ_1\oplus \cdots \oplus dZ_n),
   \]
   and we see that $T, \tilde{\Omega}, B_{\infty}$ satisfy (\ref{eq:commute}),(\ref{eq:ci1}).
\end{proof}

\begin{definition}
  An {\it extended generalized Okubo system} is a Pfaffian system (\ref{eq:okubopfaff}) 
  satisfying the system of equations (\ref{eq:commute}),(\ref{eq:ci1}).  
\end{definition}

\begin{remark} \label{rm:henkan}
Assume that $T$, $\tilde{\Omega}$, $B_{\infty}$ satisfy the system of equations $(\ref{eq:commute}),(\ref{eq:ci1})$. 
Then replacing $B_{\infty}$ by $B_{\infty}-\lambda I_N$ for any $\lambda\in \mathbb{C}$,
we see that $T$, ${\tilde \Omega}$, $B_{\infty}-\lambda I_N$ also satisfy
 $(\ref{eq:commute}),(\ref{eq:ci1})$.
 The corresponding transformation from (\ref{eq:okubopfaff}) into
    \begin{equation*}
       dY^{(\lambda)}=-(zI_N-T)^{-1}\bigl(dz+\tilde{\Omega}\bigr)(B_{\infty}-\lambda I_N)Y^{(\lambda)}
    \end{equation*}
    is realized by the Euler transformation
    \begin{equation*}
      Y\to Y^{(\lambda)}(z):=\frac{1}{\Gamma(\lambda)}\oint (u-z)^{\lambda-1}Y(u)du.
    \end{equation*}
\end{remark}

%
 \begin{lemma} \label{lem:isomono}
   The following two assertions hold.
   \begin{enumerate}
    \item[{\rm (i)}] An extended generalized Okubo system $(\ref{eq:okubopfaff})$ is equivalent to 
    an isomonodromic deformation of $(\ref{eq:okubo})$ in the sense of Jimbo-Miwa-Ueno {\rm \cite{JMU}}.
    
    \item[{\rm (ii)}] If a generalized Okubo system $(\ref{eq:okubo})$ is given on $\mathbb{P}^1\times \{p_0\}$ 
   for $p_0\in U\setminus \{\delta_{H_{{\rm red}}}=0\}$,
   then there exists a unique extended generalized Okubo system on $\mathbb{P}^1\times W$ 
   which coincides with the given $(\ref{eq:okubo})$ restricted on $\mathbb{P}^1\times \{p_0\}$,
   where $W\subset U\setminus \{\delta_{H_{{\rm red}}}=0\}$ is a neighborhood of $p_0$.
   \end{enumerate}
 \end{lemma}
 
 \begin{proof}
   The assertion (i) is proved in Appendix \ref{app:isomonodormicdef}.
   The assertion (ii) follows from the existence and uniqueness of solutions for initial conditions 
   of the isomonodromic deformation shown in \cite{JMU}.
 \end{proof}
 
%
%

An extended generalized Okubo system naturally induces a {\it Saito bundle} as follows.
Let us recall the definition of Saito bundle following \cite{DH,Sab}:

\begin{definition}
  Let $M$ be a complex manifold and $\pi : V\to M$ be a vector bundle on $M$.
  A {\it Saito bundle} is a $5$-tuple $(V, \nabla^V, \Phi^V, R_0^V, R_{\infty}^V)$
  consisting of a flat connection $\nabla^V$ on $V$, a $1$-form $\Phi^V\in\Omega^1(M,\mbox{End}(V))$
  and two endomorphisms $R_0^V,R_{\infty}^V\in\mbox{End}(V)$, such that the following conditions are satisfied:
  \begin{gather}
    \Phi^V\wedge\Phi^V=0,\ \ \ \ [R_0^V,\Phi^V]=0, \\
    d^{\nabla^V}\Phi^V=0,\ \ \ \ \nabla^V R_0^V+\Phi^V =[\Phi^V,R_{\infty}^V], \ \ \ \ \nabla^V R_{\infty}^V=0,
  \end{gather}
  where the $\mbox{End}(V)$-valued forms $[R,\Phi^V]$
  (with $R:=R_0^V$ or $R_{\infty}^V$), $d^{\nabla^V}\Phi^V$ and $\Phi^V\wedge\Phi^V$ are defined by:
  for any $X,Y\in \Gamma(M,\Theta_M)$,
  where $\Theta_M$ denotes the sheaf of vector fields on $M$,
  \begin{align*}
    &(\Phi^V\wedge\Phi^V)_{X,Y}:=\Phi_X^V\Phi_Y^V-\Phi_Y^V\Phi_X^V,\ \ \ \ [R,\Phi^V]_X:=[R,\Phi_X^V], \\
    &(d^{\nabla^V}\Phi)_{X,Y}:=\nabla_X^V(\Phi_Y^V)-\nabla^V_{Y}(\Phi_X^V)-\Phi_{[X,Y]}^V.
  \end{align*}
 \end{definition}

Returning to our setting, let $\mathbb{C}^N(U)$ be a trivial bundle of rank $N$ on a domain $U$ 
and $({\bf e}_1,\dots,{\bf e}_N)^t$ be a basis of sections of $\mathbb{C}^N(U)$.
Let $\nabla^{(0)}$ be the trivial connection on $\mathbb{C}^N(U)$ 
defined by $\nabla^{(0)}({\bf e}_i)=0$ $(i=1,\dots,N)$.
Then, for an extended generalized Okubo system (\ref{eq:okubopfaff}), 
Lemma~\ref{lem:lemma2.2} implies that the $5$-tuple $(\mathbb{C}^N(U),\nabla^{(0)},\tilde{\Omega},T,-B_{\infty})$
defines a Saito bundle on $U$.

\section{Saito structure (without metric)} \label{sec:Saito structure}

\subsection{Review on Saito structure (without metric)} \label{subsec:Saito}
In this section, we review Saito structure (without metric) introduced by C. Sabbah \cite{Sab}.
Proofs of many of statements in this section are found in the literature (\cite{Sab,KMS,KM}).
In the sequel, we abbreviate Saito structure (without metric) to Saito structure for brevity.

\begin{definition}[C. Sabbah \cite{Sab}] \label{def:Saito}
 Let $M$ be a complex manifold of dimension $N$,
 $TM$ be its tangent bundle,
 and $\Theta_M$ be the sheaf of holomorphic sections of $TM$.
 A {\it Saito structure} on $M$ is a data consisting of $(\knabla, \Phi, e, E)$ in (i)-(iii) 
 satisfying the conditions (a), (b) below:
 \begin{description}
  \item[(i)] $\knabla$ is a flat torsion-free connection on $TM$,
  \item[(ii)] $\Phi$ is a symmetric Higgs field on $TM$,
  \item[(iii)] $e$ and $E$ are global sections (vector fields) of $TM$,
  respectively called {\it unit field} and {\it Euler field}.
 \end{description}
 \begin{description}
   \item[(a)] A meromorphic connection $\nabla$ on the bundle $\pi^{*}TM$ on $\mathbb{P}^1\times M$ defined by
   \begin{equation} \label{def:nabla}
     \nabla=\pi^{*}\knabla+\frac{\pi^*\Phi}{z}-\left(\frac{\Phi(E)}{z}+\knabla E\right)\frac{dz}{z}
   \end{equation}
   is integrable, where $\pi$ is the projection $\pi :\mathbb{P}^1\times M\rightarrow M$ 
   and $z$ is a non-homogeneous coordinate of $\mathbb{P}^1$,
   \item[(b)] the field $e$ is $\knabla$-horizontal (i.e., $\knabla (e)=0$) and satisfies $\Phi_e={\rm Id}$,
   where we regard $\Phi$ as an $\mbox{End}_{\mathcal{O}_M}(\Theta_M)$-valued $1$-form
   and $\Phi_e\in \mbox{End}_{\mathcal{O}_M}(\Theta_M)$ denotes the contraction of the vector field $e$ and the $1$-form $\Phi$. 
 \end{description} 
\end{definition}


\begin{remark}
 To the Higgs field $\Phi$ there  associates a product $\star$ on $\Theta_M$ defined by 
$X \star Y :=\Phi_{X}(Y)$ for  $X,Y\in \Theta_M$.
 The Higgs field $\Phi$ is said to be symmetric if the product $\star$ is commutative and associative.
 The condition $\Phi_e={\rm Id}$ in Definition \ref{def:Saito}
(b) implies that the field $e$ is the unit of the product $\star$.
 So we can introduce on $\Theta_M$ the structure of associative and commutative $\mathcal{O}_M$-algebra with a unit.
\end{remark}

Since the connection $\knabla$ is flat and torsion free, we can take a {\it flat coordinate system} $(t_1,\dots,t_N)$ 
such that $\knabla(\partial _{t_i})=0$, $1\leq i\leq N$,
(at least) on a simply-connected open set $U$ of $M$.
We take a flat coordinate system $(t_1,\dots,t_N)$ on $U$
and assume the following conditions:

 (C1) $e=\partial_{t_N}$,
 
 (C2) $E=w_1t_1\partial _{t_1}+\cdots+w_Nt_N\partial _{t_N}$ for $w_i\in\mathbb{C}\>(i=1,\dots,N)$,
 
 (C3) $w_N=1$ and $w_i-w_j\not\in\mathbb{Z}\setminus\{0\}$ for $i\neq j$.
\vspace{2mm}
\\
In this paper, a function $f\in\mathcal{O}_M(U)$ is said to be weighted homogeneous with a weight $w(f)\in \mathbb{C}$
if $f$ is an eigenfunction of the Euler operator: $Ef=w(f)f$.
In particular, the flat coordinates $t_i$, $1\leq i \leq N$ are weighted homogeneous with $w(t_i)=w_i$.

We fix a basis $\{\partial _{t_1},\dots,\partial _{t_N}\}$ of $\Theta_M(U)$
and introduce the following matrices:

 (i) $\tilde{\Phi}=(\tilde{\Phi}_{ij})$ is the representation matrix of $\Phi$, 
 namely the $(i,j)$-entry $\tilde{\Phi}_{ij}$ is a $1$-form defined by
 \begin{equation}\label{eq:deftildeB}
  \Phi(\partial _{t_i})=\sum_{j=1}^N\tilde{\Phi}_{ij}\partial _{t_j},
\quad 1\leq i\leq N,
 \end{equation}
 
 (ii) $\mathcal{T}=(\mathcal{T}_{ij})$ and $\mathcal{B}_{\infty}=\bigl((\mathcal{B}_{\infty})_{ij}\bigr)$ are the representation matrices of $-\Phi(E)$ and $\knabla E$ respectively,
 namely
 \begin{equation}
  -\Phi_{\partial_{t_i}}(E)=\sum_{j=1}^N\mathcal{T}_{ij}\partial_{t_j},\ \ \ \ 
  \knabla_{\partial_{t_i}}(E)=\sum_{j=1}^N(\mathcal{B}_{\infty})_{ij}\partial_{t_j}.
 \end{equation}

\begin{lemma} \label{lem:Btow}
$\mathcal{B}_{\infty}=\mbox{diag}\,(w_1,\dots,w_N)$.
\end{lemma}

%
In the following, we assume that  $-\Phi(E)$ (or equivalently $\mathcal{T}$) is generically regular on $U$ (cf. (A3) in Section~\ref{sec:okuboiso}).

\begin{remark}
  By the assumption that $-\Phi(E)$ is generically regular, we can take a coordinate system $(-z_{k,0},\dots,-z_{k,m_k-1})_{1\leq k\leq n}$
  as (A4) in Section~\ref{sec:okuboiso}, which is called a {\it canonical coordinate system}.
  The canonical coordinate system for regular $F$-manifolds are constructed by L. David and C. Hertling \cite{DH}.
  The product $\star$ is written in a simple form with respect to the canonical coordinate system.
  Indeed the product is represented with respect to $(-\partial_{z_{k,0}},\dots,-\partial_{z_{k,m_k-1}})_{1\leq k\leq n}$ by
  \[
    (-\partial_{z_{k,l}})\star (-\partial_{z_{p,q}})=\left\{\begin{array}{l} -\delta_{k,p}\partial_{z_{k,l+q}},\ \ 0\leq l+q\leq m_k-1, \\
      0, \ \ \ \ \ \ \ \ \ \ \ \ \ \ \ \ l+q\geq m_k.
      \end{array}\right.
  \]
\end{remark}

\begin{lemma} \label{lem:trivial}
 The meromorphic connection $\nabla$ in $(\mbox{{\rm a}})$ is integrable if and only if $\mathcal{T}$, $\tilde{\Phi}$
 and $\mathcal{B}_{\infty}$ are subject to the following system of equations
 \begin{gather} 
      \bigl[\mathcal{T},\tilde{\Phi}\bigr]=O, \quad
      \tilde{\Phi}\wedge\tilde{\Phi}=O, \label{eq:saitorel1} \\
      dT+\tilde{\Phi}+[\tilde{\Phi},\mathcal{B}_{\infty}]=O, \quad d\tilde{\Phi}=O. \label{eq:saitorel2}
 \end{gather}
\end{lemma}

Lemmas~\ref{lem:Btow} and~\ref{lem:trivial} imply that a Saito structure induces an extended generalized Okubo system 
and thus a Saito bundle $(\mathbb{C}^N(U),\knabla,\tilde{\Phi},\mathcal{T},-\mathcal{B}_{\infty})$
by taking a flat coordinate system on $U$.
 
 \begin{remark} \label{rem:BirkhoffOkubo}
   The meromorphic connection (\ref{def:nabla}) is written as the following Pfaffian system with respect to the flat coordinate system:
   \begin{equation} \label{eq:Birkhoff}
     d\mathcal{Y}=\left(\Bigl(-\frac{\mathcal{T}}{z}+\mathcal{B}_{\infty}\Bigr)\frac{dz}{z}-\frac{1}{z}\tilde{\Phi}\right)\mathcal{Y}.
   \end{equation}
   The system of equations (\ref{eq:saitorel1}),(\ref{eq:saitorel2}) is equivalent to the integrability condition of (\ref{eq:Birkhoff}).
   The system of ordinary linear differential equations 
   \begin{equation*}
     \frac{d\mathcal{Y}}{dz}=\left(-\frac{\mathcal{T}}{z^2}+\frac{\mathcal{B}_{\infty}}{z}\right)\mathcal{Y}
   \end{equation*}
   has an irregular singular point of Poincar\'e rank $1$ at $z=0$ and a regular singular point at $z=\infty$,
   which is called a {\it Birkhoff normal form}.
   A Birkhoff normal form can be transformed into a generalized Okubo system using a Fourier-Laplace transformation.
 \end{remark}

\begin{lemma} \label{lem:Cishomo}
 There is a unique matrix $\mathcal{C}$ whose entries are in $\mathcal{O}_M(U)$ such that 
 \begin{equation*}
  \mathcal{T}=-E\mathcal{C}, \ \ \  
  \tilde{\Phi}=d\mathcal{C}
 \end{equation*} 
 and that each matrix entry $\mathcal{C}_{ij}$ of $\mathcal{C}$ is weighted homogeneous with $w(\mathcal{C}_{ij})=1-w_i+w_j$. 
\end{lemma}

\begin{lemma} \label{lem:flatC=t}
 Let $(t_1,\dots,t_N)$ be a flat coordinate system of a Saito structure.
 Then it holds that $t_j=\mathcal{C}_{Nj}=-w_j^{-1}\mathcal{T}_{Nj}$, $1\leq j\leq N.$
\end{lemma}

\begin{proof}
 Since
 \[
    \partial_{t_i}=\partial_{t_i}\star e=\sum_{j=1}^N(\tilde{\Phi}_{\partial_{t_i}})_{Nj}\partial_{t_j},
 \]
 we have $\tilde{\Phi}_{Nj}=\sum_{i=1}^Ndt_i$,
 which implies $\mathcal{C}_{Nj}=t_j$.
\end{proof}

\begin{proposition}[Konishi-Minabe \cite{KM}] \label{prop:potential vector field}
  There is a unique $N$-tuple of holomorphic functions $\vec{g}=(g_1,\dots,g_N)\in \mathcal{O}_M^N(U) $
  such that $\mathcal{C}_{ij}=\frac{\partial g_j}{\partial t_i}$,
 and that $g_j$ is weighted homogeneous with $w(g_j)=1+w_j$.
  The vector $\vec{g}$ (or precisely the vector field $\mathcal{G}=\sum_{i=1}^Ng_i\partial_{t_i}$)
is called a potential vector field.
\end{proposition}

\begin{remark} \label{rem:pvfC}
  It is readily found that the potential vector field $\vec{g}=(g_1,\dots,g_N)$ is explicitly given by
  \begin{equation} \label{eq:pvfC}
    g_j=\frac{1}{1+w_j}\sum_{i=1}^Nw_it_i\mathcal{C}_{ij},\ \ \ 1\leq j\leq N.
  \end{equation}
\end{remark}


\begin{proposition} \label{prop:genWDVV}
 The potential vector field $\vec{g}=(g_1,\dots,g_N)$ is a solution to the following system of nonlinear differential equations:
 \begin{gather}
  \sum_{m=1}^N\frac{\partial^2 g_m}{\partial t_k\partial t_i}\frac{\partial^2 g_j}{\partial t_l\partial t_m}=
  \sum_{m=1}^N\frac{\partial^2 g_m}{\partial t_l\partial t_i}\frac{\partial^2 g_j}{\partial t_k\partial t_m},
\ \ i,j,k,l=1,\dots,N, \label{eq:genWDVV} \\
  \frac{\partial ^2 g_j}{\partial t_N\partial t_i}=\delta_{ij},\ \ \ \ i,j=1,\dots,N, \label{eq:vpunit} \\
  Eg_j=\sum_{k=1}^Nw_kt_k\frac{\partial g_j}{\partial t_k}=(1+w_j)g_j,\ \ \ j=1,\dots,N. \label{eq:vphomo}
 \end{gather}
\end{proposition}

\begin{definition} \label{def:extWDVV}
The system of non-linear differential equations
$(\ref{eq:genWDVV})$-$(\ref{eq:vphomo})$
for a vector $\vec{g}=(g_1,\ldots,g_N)$ is called the extended WDVV equation.
\end{definition}

\begin{remark}
  The system of differential equations (\ref{eq:genWDVV}) is called ``oriented associativity equations'' in \cite{LoM,Ma}.
\end{remark}

Conversely, starting with a solution to (\ref{eq:genWDVV})-(\ref{eq:vphomo}), 
it is possible to reconstruct a Saito structure.

\begin{proposition} \label{prop:fromvptoSaito}
 Take constants $w_j\in\mathbb{C},$ $1\leq j\leq N$ satisfying $w_i-w_j\not\in\mathbb{Z}$ and $w_N=1$
and
assume that $\vec{g}=(g_1,\dots,g_N)$ is a holomorphic solution to $(\ref{eq:genWDVV})$-$(\ref{eq:vphomo})$
 on a simply-connected domain $U$ in $\mathbb{C}^N$.
 Then there is a Saito structure on $U$ which has $(t_1,\dots,t_N)$ as a flat coordinate system.
\end{proposition}

\begin{proof}
  Define $E:=\sum_{i=1}^Nw_it_i\partial_{t_i},$ $e:=\partial_{t_N},$
  $\mathcal{C}_{ij}:=\frac{\partial g_j}{\partial t_i},$
  $\Phi:=d\mathcal{C}$
  and $\knabla(\partial_{t_i})=0,$ $1\leq i\leq N$.
Then $(\knabla, \Phi, E, e)$ satisfies the conditions (a), (b) in Definition \ref{def:Saito}
and $\vec{g}$ becomes its potential vector field.
\end{proof}

\subsection{Saito structure on the space of isomonodromic deformations of a generalized Okubo system} \label{subsec:Okubo}

We consider an extended generalized Okubo system (\ref{eq:okubopfaff}) with the same assumptions as in Section~\ref{sec:okuboiso}.
We show that the space $U$ of variables of (\ref{eq:okubopfaff}) is equipped with
a Saito structure under some generic condition.
The arguments below closely follow \cite[Chapter VII]{Sab}.

In a general setting, for a Saito bundle $(V,\nabla^V,\Phi^V,R_0^V,R_{\infty }^V)$ on a complex manifold $M$,
a $\nabla^V$-horizontal section $\omega$ of $V$ is said to be a {\it primitive section} if it satisfies the following conditions:

\begin{enumerate}
  \item[(i)] $R_{\infty}^V\omega=\lambda \cdot\omega$ for some $\lambda\in\mathbb{C}$,
  \item[(ii)] $\varphi_{\omega} : TM \to V$ defined by $\varphi_{\omega}(X):=\Phi_X^V(\omega)$ 
  for $X\in\Theta_M$ is an isomorphism.
\end{enumerate}

Thanks to \cite[Chapter VII, Theorem 3.6]{Sab}, if there is a primitive section $\omega$, 
we can introduce a Saito structure $(\knabla,\Phi,e,E)$ on $M$ via $\varphi_{\omega}$:
\begin{gather*}
  \knabla:=\varphi_{\omega}^{-1}\circ \nabla^V\circ \varphi_{\omega},\ \ 
  \Phi:=\varphi_{\omega}^{-1}\circ \Phi^V\circ \varphi_{\omega},\ \ 
  e:=\varphi_{\omega}^{-1}(\omega), \ \ E:=\varphi_{\omega}^{-1}(R_0^V\omega).
\end{gather*}

Returning to our setting,
we consider an extended generalized Okubo system (\ref{eq:okubopfaff}) defined on $\mathbb{P}^1\times U$,
where $U$ is a domain on $\mathbb{C}^N$.
For a while, we treat (\ref{eq:okubopfaff}) on an appropriate smaller domain $W\subset U\setminus \{\delta_{H_{{\rm red}}}=0\}$
 so that we can take an invertible matrix $P$ such that
\begin{equation} \label{eq:PtoT,B}
 P^{-1}TP=Z_1\oplus \cdots \oplus Z_n,\ \ \ \ 
 P^{-1}\tilde{\Omega}P
 =-dZ_1\oplus \cdots \oplus dZ_n.
\end{equation}
Then, as we observed in Section~\ref{sec:okuboiso}, the extended generalized Okubo system (\ref{eq:okubopfaff}) induces
a Saito bundle $(\mathbb{C}^N(W),\nabla^{(0)},\tilde{\Omega},T,-B_{\infty})$.
By definition, each of $\{{\bf e}_1,\dots,{\bf e}_N\}$ is $\nabla^{(0)}$-horizontal
and satisfies the condition (i). 

In the following, we employ the notation $i_{k,l}:=\sum_{j=1}^km_{j-1}+l+1$  for $k=1,\dots ,n, l=0,\dots ,m_k-1$,
 where we put $m_0=0$.

\begin{lemma} \label{lem:nscondflat}
 Take one of the sections $\{{\bf e}_1,\dots,{\bf e}_N\}$ of $\mathbb{C}^N(W)$, say ${\bf e}_N$.
 Then ${\bf e}_N$ is a primitive section of the Saito bundle $(\mathbb{C}^N(W),\nabla^{(0)},\tilde{\Omega},T,-B_{\infty})$
 if and only if $P_{N,i_{k,0}}\neq 0$, $1\leq k\leq n$, at any point on $W$.
\end{lemma}

\begin{proof}
 Noting ${\bf e}_N=(P_{N1},\dots,P_{NN})P^{-1}({\bf e}_1,\dots,{\bf e}_N)^t$ and (\ref{eq:PtoT,B}), we have
 \begin{align*}
   \varphi_{{\bf e}_N}(X)&=-(P_{N1},\dots,P_{NN})\Bigl(XZ_1\oplus\cdots\oplus XZ_n\Bigr)
   P^{-1}({\bf e}_1,\dots,{\bf e}_N)^t \\
   &=-(Xz_{1,0},\dots ,Xz_{n,m_n-1})
   \left(\bigoplus_{k=1}^n\sum_{l=0}^{m_k-1}P_{N,i_{k,l}}\Lambda_k^l\right) P^{-1}({\bf e}_1,\dots,{\bf e}_N)^t
 \end{align*}
 for $X\in \Theta_{U}$.
 Then it is clear that $\varphi_{{\bf e}_N}$
 is an isomorphism if and only if $\prod_{k=1}^nP_{N,i_{k,0}}\neq 0$.
\end{proof}


\begin{lemma} \label{lem:PtoT}
 For an extended generalized Okubo system $(\ref{eq:okubopfaff})$,
 the following two conditions are equivalent:
 \begin{enumerate}
   \item[{\rm (i)}] $dT_{N1}\wedge \cdots \wedge dT_{NN}\neq 0$
                  at any point on $W$,
   \item[{\rm (ii)}] $\prod_{k=1}^nP_{N,i_{k,0}}\neq 0$ at any point on $W$.
 \end{enumerate}
\end{lemma}

\begin{proof}
 From (\ref{eq:ci1}), it holds that
 \[
    dT_{Nj}=(\lambda_N-\lambda_j-1)\tilde{\Omega}_{Nj},
 \] 
 from which and (\ref{eq:taikakuBi}) we have
 \begin{align*}
   &(dT_{N1},\dots,dT_{NN})\,\mbox{diag}[\lambda_N-\lambda_1-1,\lambda_N-\lambda_2-1,\dots,-1]^{-1}P
   =(\tilde{\Omega}_{N1},\dots,\tilde{\Omega}_{NN})P \label{eq:TP=BP} \\
   &=-(dz_{1,0},\dots ,dz_{n,m_n-1})
   \left(\bigoplus_{k=1}^n\sum_{l=0}^{m_k-1}P_{N,i_{k,l}}\Lambda_k^l\right).
 \end{align*}
 Hence we obtain that (i)$\iff$(ii).
\end{proof}

So far we have treated (\ref{eq:okubopfaff}) on $W$ outside of $\{\delta_{H_{{\rm red}}}=0\}$.
From now on, we consider (\ref{eq:okubopfaff}) on $U$ including $\{\delta_{H_{{\rm red}}}=0\}$.

\begin{theorem} \label{saitojacobian}
 An extended generalized Okubo system $(\ref{eq:okubopfaff})$ induces
a Saito structure on $U$ if and only if
 \begin{equation} \label{jacobian}
     dT_{N1}\wedge \cdots \wedge dT_{NN} \neq 0 \mbox{ on } U.
 \end{equation}
 When the condition $(\ref{jacobian})$ is satisfied,
 the set of variables $t_j:=-(\lambda_j-\lambda_N+1)^{-1}T_{Nj}=C_{Nj}$, $1\leq j\leq N$ gives a flat coordinate system.
\end{theorem}

\begin{proof}
  In virtue of Lemmas~\ref{lem:nscondflat} and~\ref{lem:PtoT},
  the theorem holds on $W\subset U\setminus \{\delta_{H_{{\rm red}}}=0\}$.
  We also see that $\{t_j=C_{Nj}\}$ is a flat coordinate system by Lemma~\ref{lem:flatC=t}.
  Then $(\knabla,\Phi,E,e)$, where
  $\Phi:=dC$, $E:=\sum_{k=1}^N(\lambda_k-\lambda_n+1)t_k\partial_{t_k}$, $e:=\partial_{t_N}$ and
  $\knabla$ is the connection defined by $\knabla(\partial_{t_i})=0$, 
  satisfies the conditions (a),(b) of Saito structure on $W$.
  Due to the identity theorem,
  $(\knabla,\Phi,E,e)$ satisfies the conditions (a),(b) on $U$.
  Hence (\ref{eq:okubopfaff}) induces a Saito structure on $U$.
\end{proof}

\begin{remark} \label{rem:uniqunessofSaitoOkubo}
  In general, the Saito structure induced by an extended generalized Okubo system is not unique:
  there are $N$ choices of eigenvectors $\{{\bf e}_1,\dots, {\bf e}_N\}$ of $B_{\infty}$,
  correspondingly we have at most $N$ distinct Saito structures associated with the extended generalized Okubo system.
  By designating one of eigenvectors of $B_{\infty}$,
  we can specify a unique Saito structure.
\end{remark}

Applying Theorem~\ref{saitojacobian},
we can solve the initial value problem on regular Saito structures for generic initial conditions.

\begin{corollary} \label{cor:initial}
  The following two assertions hold:
  \begin{enumerate}
  
  \item[{\rm (i)}] Let $S$ be a constant $N\times N$ matrix that is regular in the sense of {\rm (A3)} in Section~$\ref{sec:okuboiso}$,
  and $R$ be a constant $N\times N$ matrix 
  whose eigenvalues $\{\lambda_1,\dots,\lambda_N\}$ satisfy 
  $\lambda_i-\lambda_j\not\in \mathbb{Z}\setminus\{0\}$ for $i\neq j$
  and $v$ be one of eigenvectors of $R$.
  For $(S,R,v)$ and a point $p_0\in \mathbb{C}^N$,
  if $S$ is generic (the meaning of that ``$S$ is generic'' is clarified in the proof),
  there exists a regular Saito structure on a neighborhood of $p_0$ uniquely up to isomorphisms.
  
  \item[{\rm (ii)}] 
  Let $(\knabla, \Phi, E, e)$ be a regular Saito structure on a domain $U$ and $p_0\in U\setminus\{\delta_{H_{red}}=0\}$
  for $H(z)=(z\cdot id_{TU}-\mathcal{T})$ (see {\rm (A3)} in Section~$\ref{sec:okuboiso}$ for the definition of $\delta_{H_{red}}$).
  Then the Saito structure $(\knabla, \Phi, E, e)$
  is uniquely determined up to isomorphisms
  by $-\Phi(E)|_{p_0}, \knabla(E)|_{p_0}\in \mbox{End}_{\mathbb{C}}(T_{p_0}X)$.
  \end{enumerate}
\end{corollary}

\begin{proof}
  For a given $(S,R,v)$ satisfying the assumption in (i), let $\lambda_N$ be the eigenvalue of $v$.
  Take an invertible matrix $G$ such that its $N$-th column coincides with $v$ 
  and such that $G^{-1}RG=\mbox{diag}(\lambda_1,\dots,\lambda_N)$.
  We consider the following generalized Okubo system:
  \begin{equation} \label{eq:initialOkubo}
    (zI_N-T)\frac{dY}{dz}=-B_{\infty}Y,
  \end{equation}
  where we put $T:=G^{-1}SG$ and $B_{\infty}:=\mbox{diag}(\lambda_1,\dots,\lambda_N)$.
  Then in virtue of Lemma~\ref{lem:isomono},
  there exists a unique extended generalized Okubo system on $\mathbb{P}^1\times W$ determined by (\ref{eq:initialOkubo}),
  where $W\subset U\setminus \{\delta_{H_{red}}=0\}$ is a neighborhood of $p_0$.
  If the extended generalized Okubo system satisfies the condition (\ref{jacobian}) of Theorem~\ref{saitojacobian}
  (which is the meaning of ``generic condition for $S$'' mentioned in the assertion (i)),
  there exists a unique Saito structure on $W$ determined by $(S,R,v)$.
  
  We prove (ii).
  The unit field $e|_{p_0}\in T_{p_0}X$ at $p_0$ is an eigenvector 
  of $\knabla(E)|_{p_0}$ belonging to the eigenvalue $w_N=1$ (i.e. $\knabla_e(E)=e$),
  and thus $e|_{p_0}$ is recovered from $\knabla(E)|_{p_0}$.
  Then, from (i), we see that the Saito structure $(\knabla, \Phi, E, e)$ is uniquely determined by 
  $(-\Phi(E)|_{p_0}, \knabla(E)|_{p_0}, e|_{p_0})$.
\end{proof}

\begin{remark} \label{rem:initialFS}
  Saito structure is an upper structure of $F$-manifold introduced by C. Hertling and Y. Manin \cite{HM,He}.
  The initial condition theorem for regular $F$-manifolds was studied by L. David and C. Hertling \cite{DH}.
  In particular, they proved that an $F$-manifold is uniquely determined up to isomorphisms 
  by $-\Phi(E)|_{p_0}$.
\end{remark}

\section{Potential vector fields associated with solutions to the Painlev\'e equations} \label{sec:flatPainleve}

In this section, applying the results in the previous section,
we introduce potential vector fields associated with solutions 
to the (classical) Painlev\'e equations.

In \cite{KMS}, the following was proved:

\begin{proposition}[\cite{KMS}] \label{thm:genPVI}
  In the case of $N=3$, there is a correspondence between 
  generic solutions to the sixth Painlev\'e equation PVI and
  generic solutions to the extended WDVV equation
 \begin{gather}
  \sum_{m=1}^3\frac{\partial^2 g_m}{\partial t_k\partial t_i}\frac{\partial^2 g_j}{\partial t_l\partial t_m}=
  \sum_{m=1}^3\frac{\partial^2 g_m}{\partial t_l\partial t_i}\frac{\partial^2 g_j}{\partial t_k\partial t_m},
\ \ i,j,k,l=1,2,3, \label{eq:genWDVV2} \\
  \frac{\partial ^2 g_j}{\partial t_3\partial t_i}=\delta_{ij},\ \ \ \ i,j=1,2,3, \label{eq:vpunit2} \\
  Eg_j=\sum_{k=1}^3w_kt_k\frac{\partial g_j}{\partial t_k}=(1+w_j)g_j,\ \ \ j=1,2,3 \label{eq:vphomo2}
 \end{gather}
 satisfying the additional condition
  \begin{equation} \label{eq:regcondPVI}
   \left(-(1+w_j-w_i)\displaystyle\frac{\partial g_j}{\partial t_i}\right)_{1\leq i,j\leq 3}\sim 
   \begin{pmatrix} z_{1,0} & {} & {} \\ {} & z_{2,0} & {} \\ {} & {} & z_{3,0} \end{pmatrix}.
 \end{equation}
\end{proposition}

\begin{remark}
  A. Arsie and P. Lorenzoni \cite{AL0,Lo} showed that three-dimensional regular semisimple bi-flat $F$-manifolds 
  is parameterized by solutions to the Painlev\'e VI equation.
\end{remark}

Regarding the Painlev\'e V equation, we obtain the following result 
as a consequence of the arguments in the present paper:

\begin{proposition}[Painlev\'e V]
 In the case of $N=3$, there is a correspondence between 
 generic solutions to the fifth Painlev\'e equation PV and
 generic solutions to the extended WDVV equation 
 $(\ref{eq:genWDVV2})$-$(\ref{eq:vphomo2})$
 satisfying the additional condition
  \begin{equation} \label{eq:addconPV}
   \left(-(1+w_j-w_i)\displaystyle\frac{\partial g_j}{\partial t_i}\right)_{1\leq i,j\leq 3}\sim 
   \begin{pmatrix} z_{1,0} & z_{1,1} & {} \\ {} & z_{1,0} & {} \\ {} & {} & z_{2,0} \end{pmatrix}.
 \end{equation}
\end{proposition}
%

\begin{proof}
According to \cite{JM,KawDT}, the Painlev\'e V equation is derived from the isomonodromic deformation 
of the system of linear differential equations which has regular singular points at $z=0, \infty$ and
an irregular singular point of Poincar\'e rank 1 at $z=1$:
\begin{equation} \label{eq:rank2linPV}
\frac{dY}{dz}=\left( \frac{A_0}{z}+\frac{A_1}{(z-1)^2}+\frac{A_2}{z-1} \right)Y,
\end{equation}
where $A_0, A_1, A_2$ are $2\times 2$ matrices independent of $z$.
On (\ref{eq:rank2linPV}), we assume the following conditions:
 \begin{enumerate}
  \item[(i)] the exponent matrix at $z=0$ is given by $T^{(0)}_0=\mbox{diag}(\theta^0,0)$,
  
  \item[(ii)] $A_{\infty}=\mbox{diag}(\theta^\infty_1,\theta^\infty_2)$ for $A_{\infty}:=-A_0-A_2$,
  
  \item[(iii)] a fundamental system of formal solutions at $z=1$ is given as
\begin{align}
  &Y(z)=G^{(1)}(1+\hat{Y}^{(1)}_1(z-1)+\cdots)e^{T^{(1)}(z)}, \\
  &T^{(1)}(z)=\begin{pmatrix} t & {} \\ {} & 0 \end{pmatrix}\frac{-1}{z-1}+\begin{pmatrix} \theta^1 & {} \\ {} & 0 \end{pmatrix}\log (z-1),
\end{align}
where $G^{(1)}$ is an invertible matrix such that $(G^{(1)})^{-1}A_1G^{(1)}=\mbox{diag}(t,0)$.
 \end{enumerate}
Under these assumptions, we find that the matrices $A_0, A_1, A_2$ are written as
\begin{align*}
A_0&=
\begin{pmatrix}
u & 0 \\
0 & 1
\end{pmatrix}^{-1}
\left\{
\frac{1}{\theta^\infty_{12}}
\begin{pmatrix}
pq+\theta^\infty_{12} \\
-p
\end{pmatrix}
\begin{pmatrix}
p(q-1)+\theta^0 & (pq+\theta^\infty_{12})(q-1)+\theta^0 q
\end{pmatrix}
\right\}
\begin{pmatrix}
u & 0 \\
0 & 1
\end{pmatrix}, \\
A_1&=
\begin{pmatrix}
u & 0 \\
0 & 1
\end{pmatrix}^{-1}
\left\{
-\frac{t}{\theta^\infty_{12}}
\begin{pmatrix}
(pq-\theta^\infty_2)(q-1)-\theta^\infty_1 \\
p(1-q)+\theta^\infty_2
\end{pmatrix}
\begin{pmatrix}
1 & q
\end{pmatrix}
\right\}
\begin{pmatrix}
u & 0 \\
0 & 1
\end{pmatrix}, \\
A_2&=-A_0-
\begin{pmatrix}
\theta^\infty_1 & 0 \\
0 & \theta^\infty_2
\end{pmatrix},
\end{align*}
and $\theta^\infty_{12}:=\theta^\infty_1-\theta^\infty_2$,
where $t$ is a deformation parameter, $p, q$ are unknown functions for the isomonodromic deformation, 
$\theta^0,\theta^1,\theta_1^{\infty},\theta_2^{\infty}$ are constant parameters 
satisfying $\theta^0+\theta^1+\theta_1^{\infty}+\theta_2^{\infty}=0$
and $u$ is an overall parameter.
The Painlev\'e V equation is equivalent to the Hamiltonian system
\begin{equation} \label{eq:sysH_V}
  \frac{dq}{dt}=\frac{\partial H_{\mathrm{V}}}{\partial p},\ \ 
  \frac{dp}{dt}=-\frac{\partial H_{\mathrm{V}}}{\partial q}
\end{equation}
 with the Hamiltonian
\begin{equation}
tH_{\mathrm{V}}=p(p+t)q(q-1)+(\theta^0+\theta^\infty_1-\theta^\infty_2)qp
+(\theta^\infty_2-\theta^\infty_1)p-\theta^\infty_2 tq.
\end{equation}
Indeed, after the canonical transformation 
\begin{equation}
  (q,p) \to (\lambda, \mu):=\left(1-\frac{1}{q}, q(pq-\theta_2^{\infty})\right),
\end{equation}
$\lambda$ satisfies the Painlev\'e V equation
\begin{align*}
\frac{d^2\lambda}{dt^2}=\left( \frac{1}{2\lambda}+\frac{1}{\lambda-1} \right)\left( \frac{d\lambda}{dt} \right)^2-\frac{1}{t}\frac{d\lambda}{dt}+
\frac{(\lambda-1)^2}{t^2}\left( \alpha \lambda+\frac{\beta}{\lambda} \right)
+\gamma\frac{\lambda}{t}-\frac12\frac{\lambda(\lambda+1)}{\lambda-1},
\end{align*}
with
\[
\alpha=\frac{(\theta^\infty_1-\theta^\infty_2)^2}{2}, \quad \beta=-\frac{(\theta^0)^2}{2}, \quad \gamma=1-\theta^1.
\]

Now we prove the proposition.
At the first step, we start from a solution to the extended WDVV equation (\ref{eq:genWDVV2})-(\ref{eq:vphomo2}) 
satisfying the condition (\ref{eq:addconPV}).
Then we have a $3\times 3$ extended generalized Okubo system (\ref{eq:okubopfaff}) with
 \[
  T_{ij}=-(1+w_j-w_i)\frac{\partial g_j}{\partial t_i},\ \ 
  \tilde{\Omega}_{ij}=\sum_{k=1}^3\frac{\partial^2 g_j}{\partial t_k\partial t_i}dt_k,\ \ 
  B_{\infty}=\mbox{diag}(w_1,w_2,w_3).
 \]
The condition (\ref{eq:addconPV}) implies that there is an invertible matrix $P$ such that
 \begin{equation} \label{eq:PVTsim}
   P^{-1}TP= 
   \begin{pmatrix} z_{1,0} & z_{1,1} & {} \\ {} & z_{1,0} & {} \\ {} & {} & z_{2,0} \end{pmatrix}.
 \end{equation}
Then, by the procedure explained in Appendix~\ref{app:compOkubo},
we can reduce the extended generalized Okubo system to an isomonodromic deformation 
of the following $2\times 2$ system of linear differential equations:
\begin{equation} \label{eq:rank2linPVOkubored}
\frac{d\bar{Y}}{d\bar{z}}=\left( \frac{\bar{A}_0}{\bar{z}-z_{2,0}}+\frac{\bar{A}_1}{(\bar{z}-z_{1,0})^2}+\frac{\bar{A}_2}{\bar{z}-z_{1,0}} \right)\bar{Y}.
\end{equation}
Change the variable $\bar{z}\to z=(\bar{z}-z_{2,0})/(z_{1,0}-z_{2,0})$.
Then (\ref{eq:rank2linPVOkubored}) is changed to (\ref{eq:rank2linPV}) satisfying the conditions (i)-(iii).
Hence we obtain a solution to (\ref{eq:sysH_V}),
which is equivalent to a solution to PV.
Particularly we have the following correspondence between parameters:
for the matrix $P$ in (\ref{eq:PVTsim}), put $B:=-P^{-1}B_{\infty}P$.
Then
\[
  t=\frac{z_{1,1}B_{21}}{z_{1,0}-z_{2,0}},\ \ 
  \theta^0=B_{33},\ \ \theta^1=B_{11}+B_{22},\ \ \theta^{\infty}_1=w_1-w_3,\ \ \theta^{\infty}_2=w_2-w_3.
\]

Next we start from an isomonodromic deformation of the system (\ref{eq:rank2linPV}) with the conditions (i)-(iii)
(which is equivalent to giving a generic solution to the Painlev\'e V equation).
The system (\ref{eq:rank2linPV}) can be transformed into a generalized Okubo system of rank three 
applying the procedure explained in Appendix~\ref{app:compOkubo} (cf. \cite{KawDT}):
\begin{equation} \label{eq:rank3linPV}
(z I_3-S_\mathrm{V})\frac{d\Psi}{dz}=C_\mathrm{V}\Psi,
\end{equation}
where
\begin{align*}
S_\mathrm{V}&=
\begin{pmatrix}
1 & 1 & 0 \\
0 & 1 & 0 \\
0 & 0 & 0
\end{pmatrix}, \\
C_\mathrm{V}&=
\begin{pmatrix}
\theta^1 & -\frac{1}{t}\det A_2 & -\frac{1}{t}(pq-\theta^1-\theta^\infty_2)\\
t & 0 & 1 \\
t(pq(q-1)-\theta^1-\theta^\infty_1-\theta^\infty_2 q) & (C_{\mathrm{V}})_{32} & \theta^0
\end{pmatrix},
\end{align*}
and
\[
(C_{\mathrm{V}})_{32}=
(q-1)(q(q-1)p^2+(\theta^\infty_2-\theta^\infty_1-(\theta^1+2\theta^\infty_2)q)p+\theta^\infty_2(\theta^1+\theta^\infty_2)).
\]
Note that $S_\mathrm{V}$ is regular in the sense of (A3) in Section~\ref{sec:okuboiso}.
Change the variables $z\to \bar{z}=(z_{1,0}-z_{2,0})z+z_{2,0}$ and $t\to z_{1,1}=(z_{1,0}-z_{2,0})t$.
Diagonalizing $C_{\mathrm{V}}$ by applying a gauge transformation to (\ref{eq:rank3linPV}),
we obtain from $S_\mathrm{V}$ a matrix $T$ satisfying (\ref{eq:PVTsim}). 
If the resulting matrix $T$ satisfies the condition (\ref{jacobian}) in Theorem~\ref{saitojacobian} (which is a generic condition),
 we obtain a flat coordinate system and a potential vector field
satisfying the condition (\ref{eq:addconPV}).
\end{proof}

We can obtain similar results on the remaining Painlev\'e equations except for Painlev\'e I:

\begin{proposition}[Painlev\'e IV]
 In the case of $N=3$, there is a correspondence between 
 generic solutions to the fourth Painlev\'e equation PIV and
 generic solutions to the extended WDVV equation $(\ref{eq:genWDVV2})$-$(\ref{eq:vphomo2})$ 
 satisfying the additional condition
  \begin{equation} \label{eq:addconPIV}
    \left(-(1+w_j-w_i)\displaystyle\frac{\partial g_j}{\partial t_i}\right)_{1\leq i,j\leq 3}\sim 
   \begin{pmatrix} z_{1,0} & z_{1,1} & z_{1,2} \\ {} & z_{1,0} & z_{1,1} \\ {} & {} & z_{1,0} \end{pmatrix}.
 \end{equation}
\end{proposition}
 
 \begin{proof}
 The Painlev\'e IV equation is derived from the isomonodromic deformation of the following $2\times 2$
 system of linear differential equations which has a regular singular point at $z=\infty$ and
 an irregular singular point of Poincar\'e rank 2 at $z=0$:
 \begin{equation} \label{eq:rank2linPIV}
 \frac{dY}{dz}=\left( \frac{A_0}{z^3}+\frac{A_1}{z^2}+\frac{A_2}{z} \right)Y
 \end{equation}
 where
 \begin{align*}
 A_0&=
 \begin{pmatrix}
 u & 0 \\
 0 & 1
 \end{pmatrix}^{-1}
 \left(
 \frac{1}{\theta^\infty_{12}}
 \begin{pmatrix}
 p \\
 -pq+\theta^\infty_{12}
 \end{pmatrix}
 \begin{pmatrix}
 q & 1
 \end{pmatrix}
 \right)
 \begin{pmatrix}
 u & 0 \\
 0 & 1
 \end{pmatrix}, \\
 A_1&=
 \begin{pmatrix}
 u & 0 \\
 0 & 1
 \end{pmatrix}^{-1}
 \left(
 \frac{1}{\theta^\infty_{12}}
 \begin{pmatrix}
 pq(p-q-t)-\theta^\infty_{12}p+\theta^\infty_1 q & p(p-q-t)+\theta^\infty_1 \\
 (pq-\theta^\infty_{12})(-pq+tq+\theta^\infty_{12})+(pq-\theta^\infty_1)q^2 & \theta^\infty_{12}(t-(A_1)_{11})
 \end{pmatrix}
 \right)
 \begin{pmatrix}
 u & 0 \\
 0 & 1
 \end{pmatrix}, \\
 A_2&=-
 \begin{pmatrix}
 \theta^\infty_1 & 0 \\
 0 & \theta^\infty_2
 \end{pmatrix},
 \end{align*}
 and $\theta^\infty_{12}:=\theta^\infty_1-\theta^\infty_2$.
 The Painlev\'e IV equation is equivalent to the Hamiltonian system with the Hamiltonian
 \begin{equation}
 H_\mathrm{IV}=pq(p-q-t)+(\theta^\infty_2-\theta^\infty_1)p-(\theta^0+\theta^\infty_2)q.
 \end{equation}
 
 We start from a solution to the extended WDVV equation (\ref{eq:genWDVV2})-(\ref{eq:vphomo2}) 
satisfying the condition (\ref{eq:addconPIV}).
 In a way similar to the PV case, we can reduce the $3\times 3$ generalized Okubo system
 to the $2\times 2$ system (\ref{eq:rank2linPIV}),
 and thus we have a solution to the Painlev\'e IV equation.

 Conversely, starting from an isomonodromic deformation of the $2\times 2$ system (\ref{eq:rank2linPIV}),
 we can construct a $3\times 3$ generalized Okubo system  
 similarly to the case of PV (cf. \cite{KawDT,Kaw}):
 \begin{equation} \label{eq:rank3linPIV}
 (z I_3-S_\mathrm{IV})\frac{d\Psi}{dz}=C_\mathrm{IV}\Psi,
 \end{equation}
 where
 \begin{align*}
 S_\mathrm{IV}&=
 \begin{pmatrix}
 0 & 1 & 0 \\
 0 & 0 & 1 \\
 0 & 0 & 0
 \end{pmatrix}, \\
 C_\mathrm{IV}&=
 \begin{pmatrix}
 0 & (q+t)(pq-\theta^\infty_1) & -p(q+t)(pq-\theta^\infty_1+\theta^\infty_2) \\
 0 & pq-\theta^\infty_1 & -p(pq-\theta^\infty_1+\theta^\infty_2) \\
 1 & -t & -pq-\theta^\infty_2
 \end{pmatrix}.
 \end{align*}
 Noting that $S_\mathrm{IV}$ satisfies the regularity condition which is equivalent to (\ref{eq:addconPIV}),
 we obtain a potential vector field satisfying the condition (\ref{eq:addconPIV})
 by a manner similar to the PV case.
\end{proof}
 
\begin{proposition}[Painlev\'e III]
 In the case of $N=4$, there is a correspondence 
 between generic solutions to the third Painlev\'e equation PIII
 and generic solutions to the extended WDVV equation 
  \begin{gather}
  \sum_{m=1}^4\frac{\partial^2 g_m}{\partial t_k\partial t_i}\frac{\partial^2 g_j}{\partial t_l\partial t_m}=
  \sum_{m=1}^4\frac{\partial^2 g_m}{\partial t_l\partial t_i}\frac{\partial^2 g_j}{\partial t_k\partial t_m},
\ \ i,j,k,l=1,2,3,4, \label{eq:genWDVV2,N=4} \\
  \frac{\partial ^2 g_j}{\partial t_4\partial t_i}=\delta_{ij},\ \ \ \ i,j=1,2,3,4, \label{eq:vpunit2,N=4} \\
  Eg_j=\sum_{k=1}^4w_kt_k\frac{\partial g_j}{\partial t_k}=(1+w_j)g_j,\ \ \ j=1,2,3,4 \label{eq:vphomo2,N=4}
 \end{gather}
 with the constraint on the weight $w_1=w_2, w_3=w_4$ and the additional condition
  \begin{equation} \label{eq:addconPIII}
     \left(-(1+w_j-w_i)\displaystyle\frac{\partial g_j}{\partial t_i}\right)_{1\leq i,j\leq 4}\sim 
     \begin{pmatrix} z_{1,0} & z_{1,1} & {} & {}\\ {} & z_{1,0} & {} & {} \\ {} & {} & z_{2,0} & z_{2,1} \\ {} & {} & {} & z_{2,0} \end{pmatrix}.
 \end{equation}
\end{proposition}

 \begin{proof}
    The Painlev\'e III equation is derived from the isomonodromic deformation 
    of the following $2\times 2$ system of linear differential equations 
    which has irregular singular points of Poincar\'e rank $1$ at $z=0,\infty$:
     \begin{equation} \label{eq:rank2linPIII}
    \frac{dY}{dz}=\left(\frac{A_2}{z^2}+\frac{A_1}{z}+A_0\right)Y,
  \end{equation}
  where 
  \begin{align*}
    &A_0=\begin{pmatrix} -1 & 0 \\ 0 & 0 \end{pmatrix},\ \ \ 
    A_1=\begin{pmatrix} u & 0 \\ 0 & 1 \end{pmatrix}^{-1}
             \begin{pmatrix} -\theta_1^{\infty} & -q \\ -r & -\theta_2^{\infty} \end{pmatrix}
             \begin{pmatrix} u & 0 \\ 0 & 1 \end{pmatrix},  \\
     &A_2=\begin{pmatrix} u & 0 \\ 0 & 1 \end{pmatrix}^{-1}\begin{pmatrix} 1 \\ p \end{pmatrix}
             \begin{pmatrix} t(1-p) & t \end{pmatrix}\begin{pmatrix} u & 0 \\ 0 & 1 \end{pmatrix}
  \end{align*}
  and $r=(pq-\theta_2^{\infty})(p-1)+\theta_1^{\infty}p$.
  The Painlev\'e III equation is equivalent to the Hamilton system with the Hamiltonian
  \[
    tH_\mathrm{III}=p^2q^2-(q^2-(\theta_1^{\infty}-\theta_2^{\infty})q-t)p+\theta_2^{\infty}q.
  \]
  Note that (\ref{eq:rank2linPIII}) has no regular singular point,
  whereas any generalized Okubo system necessarily has a regular singular point at $\infty$.
  We change the variables $z\rightarrow \xi=\frac{z}{z-1}$ and $Y\rightarrow Z=(\xi-1)^{\theta_2^{\infty}}Y$
  in order to  add a regular singularity at $\infty$ to (\ref{eq:rank2linPIII}).
  Then (\ref{eq:rank2linPIII}) is changed into
  \begin{equation} \label{eq:rank2linPIIIZ}
    \frac{dZ}{d\xi}=\left(-\frac{A_2}{\xi^2}+\frac{A_1}{\xi}-\frac{A_0}{(\xi-1)^2}-\frac{A_1+\theta_2^{\infty}I_2}{\xi-1}\right)Z,
  \end{equation}
  which has irregular singular points of Poincar\'e rank $1$ at $\xi=0,1$ 
  and a regular singular point at $\xi=\infty$.
  The residue matrix of (\ref{eq:rank2linPIIIZ}) at $\xi=\infty$ reads $\theta_2^\infty I_2$.
  
  We start from a solution to the extended WDVV equation $(\ref{eq:genWDVV2,N=4})$-$(\ref{eq:vphomo2,N=4})$
  satisfying the condition (\ref{eq:addconPIII}) with the constraint $w_1=w_2, w_3=w_4$.
  Then we have a $4\times 4$ extended generalized Okubo system with $B_{\infty}=\mbox{diag}(w_1,w_1,w_3,w_3)$.
  Applying the procedure explained in Appendix~\ref{app:compOkubo},
  we can reduce the $4\times 4$ generalized Okubo system to the $2\times 2$ system (\ref{eq:rank2linPIIIZ}), 
  which is equivalent to (\ref{eq:rank2linPIII}).
  Hence we obtain a solution to the Painlev\'e III equation.
  
  Conversely, we start from an isomonodromic deformation of the $2\times 2$ system (\ref{eq:rank2linPIII}),
  which is equivalent to an isomonodromic deformation of (\ref{eq:rank2linPIIIZ}).
  We can construct from the $2\times 2$ system (\ref{eq:rank2linPIIIZ}) a $4\times 4$ generalized Okubo system by the procedure 
  explained in Appendix \ref{app:compOkubo},
  which is determined by the following data consisting of three matrices 
  $\{S_\mathrm{III}, G_\mathrm{III}, B_{\infty}\}$:
  \begin{align*}
    &S_\mathrm{III}=\begin{pmatrix} 0 & 1 & 0 & 0 \\ 0 & 0 & 0 & 0 \\ 0 & 0 & 1 & 1 \\ 0 & 0 & 0 & 1 \end{pmatrix},\ \ 
    G_\mathrm{III}=\begin{pmatrix} \theta_1^{\infty}u & q & -pq-\theta_1^{\infty}+\theta_2^{\infty} & \frac{q(pq-\theta_2^{\infty})}{t} \\ 
                  (1-p)tu & t &-t & pq \\
                  0 & q/u & -pq/u & \frac{q(pq-\theta_2^{\infty})}{tu}  \\
                  1 & 0 & -1/u & 0 \end{pmatrix}, \\
     &B_{\infty}=\mbox{diag}(\theta_2^{\infty},\theta_2^{\infty},0,0).
  \end{align*}
  Note that $S_\mathrm{III}$ satisfies the regularity condition which is equivalent to the condition (\ref{eq:addconPIII}),
  and $B_{\infty}$ satisfies the constraint $w_1=w_2, w_3=w_4$.
  Hence we obtain a potential vector field satisfying the desired conditions.
 \end{proof}
 
\begin{proposition}[Painlev\'e II]
 In the case of $N=4$, there is a correspondence 
 between generic solutions to the second Painlev\'e equation PII
 and generic solutions to the extended WDVV equation $(\ref{eq:genWDVV2,N=4})$-$(\ref{eq:vphomo2,N=4})$ 
 with the constraint on the weight $w_1=w_2, w_3=w_4$ and the additional condition
  \begin{equation} \label{eq:addconPII}
     \left(-(1+w_j-w_i)\displaystyle\frac{\partial g_j}{\partial t_i}\right)_{1\leq i,j\leq 4}\sim 
   \begin{pmatrix} z_{1,0} & z_{1,1} & z_{1,2} & z_{1,3} \\ {} & z_{1,0} & z_{1,1} & z_{1,2} \\ {} & {} & z_{1,0} & z_{1,1} \\ {} & {} & {} & z_{1,0}
   \end{pmatrix}.
 \end{equation}
\end{proposition}

  \begin{proof}
        The Painlev\'e II equation is derived from isomonodromic deformations of the following $2\times 2$ system of linear differential equations 
    which has an irregular singular point of Poincar\'e rank $3$ at $z=\infty$:
%
\begin{equation} \label{eq:rank2linPII}
\frac{dY}{dz}=\left( A_0 z^2+A_1 z+A_2 \right)Y,
\end{equation}
where
\begin{align*}
&A_0=
\begin{pmatrix}
0 & 0 \\
0 & 1
\end{pmatrix},\quad
A_1=
\begin{pmatrix}
u & 0 \\
0 & 1
\end{pmatrix}^{-1}
\begin{pmatrix}
0 & 1 \\
p & 0
\end{pmatrix}
\begin{pmatrix}
u & 0 \\
0 & 1
\end{pmatrix},\\
&A_2=
\begin{pmatrix}
u & 0 \\
0 & 1
\end{pmatrix}^{-1}
\begin{pmatrix}
p & -q \\
pq-\theta^\infty_2 & -p+t
\end{pmatrix}
\begin{pmatrix}
u & 0 \\
0 & 1
\end{pmatrix}.
\end{align*}
The Painlev\'e II equation is equivalent to the Hamiltonian system with the Hamiltonian
\begin{equation}
H_\mathrm{II}=p^2-(q^2+t)p+\theta^\infty_2 q.
\end{equation}

By a reason similar to PIII,
we change the variables $z \to \xi=1/z$ and $Y \to Z=\xi^{\theta^\infty_2}Y$.
Then (\ref{eq:rank2linPII}) is changed into
\begin{equation} \label{eq:rank2linPIIZ}
\frac{dZ}{d\xi}=\left(
-\frac{A_0}{\xi^4}-\frac{A_1}{\xi^3}-\frac{A_2}{\xi^2}+\frac{\theta^\infty_2 I_2}{\xi}
\right)Z,
\end{equation}
  which has an irregular singular point of Poincar\'e rank $3$ at $\xi=0$ and a regular singular point at $\xi=\infty$.
  The residue matrix of (\ref{eq:rank2linPIIZ}) at $\xi =\infty$ reads $-\theta_2^\infty I_2$.
  
  It is similar to PIII that, starting a solution to the extended WDVV equation $(\ref{eq:genWDVV2,N=4})$-$(\ref{eq:vphomo2,N=4})$
  satisfying the condition (\ref{eq:addconPII}) with the constraint $w_1=w_2, w_3=w_4$,
  we obtain a solution to the Painlev\'e II equation.
  
  Conversely, we start from an isomonodromic deformation of the $2\times 2$ system (\ref{eq:rank2linPII}),
  which is equivalent to an isomonodromic deformation of (\ref{eq:rank2linPIIZ}).
  We can construct from the $2\times 2$ system (\ref{eq:rank2linPIIZ}) a $4\times 4$ generalized Okubo system by the procedure 
  explained in Appendix \ref{app:compOkubo},
  which is determined by the following data consisting of three matrices $\{S_\mathrm{II}, G_\mathrm{II}, B_{\infty}\}$:
  \begin{align*}
S_\mathrm{II}&=
\begin{pmatrix}
0 & 1 & 0 & 0 \\
0 & 0 & 1 & 0 \\
0 & 0 & 0 & 1 \\
0 & 0 & 0 & 0
\end{pmatrix}, \quad
G_\mathrm{II}=
\begin{pmatrix}
-qu & \frac{q}{\theta^\infty_2}(q^2-p+t)+1 & 0 & q(p-q^2-t) \\
u & \frac{1}{\theta^\infty_2}(p-q^2-t) & 0 & q^2-p+t \\
-\frac{pu}{\theta^\infty_2} & q/\theta^\infty_2 & 1 & 0 \\
0 & -1/\theta^\infty_2 & 0 & 1
\end{pmatrix}, \\
B_\infty&=\mathrm{diag}(\theta^\infty_2, \theta^\infty_2, 0, 0).
\end{align*}
 Note that $S_\mathrm{II}$ satisfies the regularity condition which is equivalent to the condition (\ref{eq:addconPII}),
  and $B_{\infty}$ satisfies the constraint $w_1=w_2, w_3=w_4$.
Hence we obtain a potential vector field satisfying the desired conditions.
%
  \end{proof}

 We may summarize the three-dimensional case as follows.
 
\begin{theorem} \label{thm:extWDVVandPVIPVPIV}
  There is a correspondence between generic solutions satisfying the regularity condition
   to the extended WDVV equation
   \begin{gather}
  \sum_{m=1}^3\frac{\partial^2 g_m}{\partial t_k\partial t_i}\frac{\partial^2 g_j}{\partial t_l\partial t_m}=
  \sum_{m=1}^3\frac{\partial^2 g_m}{\partial t_l\partial t_i}\frac{\partial^2 g_j}{\partial t_k\partial t_m},
\ \ i,j,k,l=1,2,3,  \\
  \frac{\partial ^2 g_j}{\partial t_3\partial t_i}=\delta_{ij},\ \ \ \ i,j=1,2,3,  \\
  Eg_j=\sum_{k=1}^3w_kt_k\frac{\partial g_j}{\partial t_k}=(1+w_j)g_j,\ \ \ j=1,2,3 
 \end{gather}
  and generic solutions to the Painlev\'e equations PVI,PV,PIV,
  where ``the regularity condition'' means that $\vec{g}$ satisfies one of $(\ref{eq:regcondPVI}),(\ref{eq:addconPV}),(\ref{eq:addconPIV})$.
\end{theorem}

\begin{remark}
 A. Arsie and P. Lorenzoni \cite{AL0,Lo,AL1} proved three-dimensional regular bi-flat $F$-manifolds
 are parameterized by solutions to PVI, PV, PIV.
 Theorem~\ref{thm:extWDVVandPVIPVPIV} provides another proof of it.
\end{remark}

The first Painlev\'e equation PI is derived from the isomonodromic deformation of the following $2\times 2$ system
of linear differential equations:
\begin{equation} \label{eq:rank2linPI}
\frac{dY}{dz}=\left( A_0 z^2+A_1 z+A_2 \right)Y,
\end{equation}
where
\begin{align*}
&A_0=
\begin{pmatrix}
0 & 1 \\
0 & 0
\end{pmatrix},\quad
A_1=
\begin{pmatrix}
0 & q \\
1 & 0
\end{pmatrix},\quad
A_2=
\begin{pmatrix}
-p & q^2+t \\
-q & p
\end{pmatrix}, 
\end{align*}
which has an irregular singular point of Poincar\'e rank $5/2$ at $z=\infty$.
PI is equivalent to the Hamiltonian system with the Hamiltonian $H_\mathrm{I}=p^2-q^3-tq$.

Similarly to the case of PII, we change the variables of (\ref{eq:rank2linPI}) as $z \to \xi=1/z$ and $Y \to Z=\xi^{-\lambda}Y$:
\begin{equation} \label{eq:rank2linPIZ}
\frac{dZ}{d\xi}=\left(
-\frac{A_0}{\xi^4}-\frac{A_1}{\xi^3}-\frac{A_2}{\xi^2}-\frac{\lambda I_2}{\xi}
\right)Z.
\end{equation}
We see that (\ref{eq:rank2linPIZ}) is transformed into the following $7\times 7$ generalized Okubo system
by the procedure explained in Appendix \ref{app:compOkubo} 
(which is a realization of  (\ref{eq:rank2linPIZ}) as a generalized Okubo system of minimal size):
\begin{align*}
S_\mathrm{I}&=
\begin{pmatrix}
0 & 1 & 0 & 0 \\
0 & 0 & 1 & 0 \\
0 & 0 & 0 & 1 \\
0 & 0 & 0 & 0
\end{pmatrix}
\oplus
\begin{pmatrix}
0 & 1 & 0 \\
0 & 0 & 1 \\
0 & 0 & 0
\end{pmatrix}, \\
G_\mathrm{I}&=
\begin{pmatrix}
\lambda & 0 & 0 & 0 & 0 & 0 & 0 \\
-p & q^2+t & p & \lambda & 0 & -q^2-t & 0 \\
0 & q & 0 & 0 & \lambda & -q & 0 \\
0 & 1 & 0 & 0 & 0 & -1 & 0 \\
0 & \lambda & 0 & 0 & 0 & 0 & 0 \\
-q & p & q & 0 & 0 & -p & \lambda \\
1 & 0 & -1 & 0 & 0 & 0 & 0
\end{pmatrix}, \\
B_\infty&=\mathrm{diag}(\lambda, \lambda, 0, 0, 0, 0, 0).
\end{align*}
We notice that $S_\mathrm{I}$ is not regular and thus we can not treat PI in the framework of the present paper.

\section{Unified treatment of the Painlev\'e equations} \label{sec:coalcas}

In the previous section, we related each of the Painlev\'e equations with an extended generalized Okubo system,
which is of minimal rank.
As for PVI, PV and PIV, it is possible to relate them to extended generalized Okubo systems of rank four.

\begin{proposition}[PVI]
In the case of $N=4$, there is a correspondence 
between generic solutions to the Painlev\'e VI equation
and generic solutions to the extended WDVV equation $(\ref{eq:genWDVV2,N=4})$-$(\ref{eq:vphomo2,N=4})$ 
with the constraint $w_1=w_2, w_3=w_4$ and the additional condition
  \begin{equation} \label{eq:addconPVI4}
    \left(-(1+w_j-w_i)\displaystyle\frac{\partial g_j}{\partial t_i}\right)_{1\leq i,j\leq 4}\sim 
     \begin{pmatrix} z_{1,0} & {} & {} & {}\\ {} & z_{2,0} & {} & {} \\ {} & {} & z_{3,0} & {} \\ {} & {} & {} & z_{4,0} \end{pmatrix}.
 \end{equation}
 \end{proposition}
 
 \begin{proof}
  PVI is derived from the isomonodromic deformation of the following $2\times 2$ system of linear differential equations:
\begin{equation} \label{eq:diffP6}
   \frac{dY}{dx}=\left(\frac{A_1}{x}+\frac{A_2}{x-1}+\frac{A_3}{x-t}\right)Y,
\end{equation}
where we assume $\det A_1=\det A_2=\det A_3=0$ without loss of generality.
We add a regular singularity to  (\ref{eq:diffP6}):
for $t_2\in\mathbb{C}\setminus\{0,1\}$,
we change the variables $x\rightarrow \xi=\frac{t_2x}{x+t_2-1}$ and $Y\rightarrow Z=(\xi-t_2)^{-\lambda}Y$,
where $\lambda\in\mathbb{C}\setminus \{0\}$ is determined so that $\det (-A_1-A_2-A_3-\lambda I_2)=0$.
Then (\ref{eq:diffP6}) is changed to
\begin{equation} \label{eq:diffP62}
  \frac{dZ}{d\xi}=\left(\frac{A_1}{\xi}+\frac{A_2}{\xi-1}+\frac{A_3}{\xi-t_1}+\frac{-A_1-A_2-A_3-\lambda I_2}{\xi-t_2}\right)Z
\end{equation}
where we put $t_1:=\frac{t\,t_2}{t+t_2-1}$.
It is apparent that (\ref{eq:diffP62}) has regular singular points at $\xi =0,1,t_1,t_2,\infty$,
and the residue matrix at $\xi=\infty$ reads $\lambda I_2$.

 We start from a solution to the extended WDVV equation $(\ref{eq:genWDVV2,N=4})$-$(\ref{eq:vphomo2,N=4})$
  satisfying the condition (\ref{eq:addconPVI4}) with the constraint $w_1=w_2, w_3=w_4$.
  Then we have a $4\times 4$ extended Okubo system with $B_{\infty}=\mbox{diag}(w_1,w_1,w_3,w_3)$.
  Applying the procedure explained in Appendix~\ref{app:compOkubo},
  we can reduce the $4\times 4$ generalized Okubo system to the $2\times 2$ system (\ref{eq:diffP62}), 
  with the change of variables 
  \[
    \xi =\frac{z-z_{1,0}}{z_{2,0}-z_{1,0}},\ \ t_1=\frac{z_{3,0}-z_{1,0}}{z_{2,0}-z_{1,0}},\ \ t_2=\frac{z_{4,0}-z_{1,0}}{z_{2,0}-z_{1,0}}.
  \]
  Hence we obtain a solution to the Painlev\'e VI equation.
  
  Conversely, we start from an isomonodromic deformation of the $2\times 2$ system (\ref{eq:diffP62}),
  which is equivalent to an isomonodromic deformation of (\ref{eq:diffP6}).
  We can construct from the $2\times 2$ system (\ref{eq:diffP62}) a $4\times 4$ generalized Okubo system by the procedure 
  explained in Appendix \ref{app:compOkubo},
  particularly we find
\[
   S_{\mathrm{VI}'}=\mbox{diag}(0,1,t_1,t_2), \ \ \ B_{\infty}=\mbox{diag}(\lambda,\lambda,0,0).
\]
  Note that $S_{\mathrm{VI}'}$ satisfies the regularity condition which is equivalent to the condition (\ref{eq:addconPVI4}),
  and $B_{\infty}$ satisfies the constraint $w_1=w_2, w_3=w_4$.
  We obtain a potential vector field satisfying the desired conditions.
 \end{proof}
 
 \begin{remark}
   The correspondence between particular 4-dimensional Frobenius manifolds and generic solutions 
   to a one-parameter family of the Painlev\'e VI equation was treated 
   by S. Romano \cite{Ro} (in a somewhat different context).
 \end{remark}
 
 We obtain similar results on PV and PIV:
 
 \begin{proposition}[PV]
 In the case of $N=4$, there is a correspondence 
 between generic solutions to the Painlev\'e V equation
 and generic solutions to the extended WDVV equation $(\ref{eq:genWDVV2,N=4})$-$(\ref{eq:vphomo2,N=4})$ 
 with the constraint $w_1=w_2, w_3=w_4$ and the additional condition
  \begin{equation} \label{eq:addconPV4}
    \left(-(1+w_j-w_i)\displaystyle\frac{\partial g_j}{\partial t_i}\right)_{1\leq i,j\leq 4}\sim 
     \begin{pmatrix} z_{1,0} & z_{1,1} & {} & {}\\ {} & z_{1,0} & {} & {} \\ {} & {} & z_{2,0} & {} \\ {} & {} & {} & z_{3,0} \end{pmatrix}.
 \end{equation}
 \end{proposition}
 
 \begin{proposition}[PIV]
 In the case of $N=4$, there is a correspondence 
 between generic solutions to the Painlev\'e IV equation
 and generic solutions to the extended WDVV equation $(\ref{eq:genWDVV2,N=4})$-$(\ref{eq:vphomo2,N=4})$ 
 with the constraint $w_1=w_2, w_3=w_4$ and the additional condition
  \begin{equation} \label{eq:addconPIV4}
    \left(-(1+w_j-w_i)\displaystyle\frac{\partial g_j}{\partial t_i}\right)_{1\leq i,j\leq 4}\sim 
     \begin{pmatrix} z_{1,0} & z_{1,1} & z_{1,2} & {}\\ {} & z_{1,0} & z_{1,1} & {} \\ {} & {} & z_{1,0} & {} \\ {} & {} & {} & z_{2,0} \end{pmatrix}.
 \end{equation}
 \end{proposition}
 
 Consequently we obtain the following theorem.
 
 \begin{theorem} \label{thm:WDVVPVI-PII}
  There is a correspondence between generic solutions to the Painlev\'e equations PII-PVI
  and generic solutions to the extended WDVV equation
   \begin{gather}
    \sum_{m=1}^4\frac{\partial^2 g_m}{\partial t_k\partial t_i}\frac{\partial^2 g_j}{\partial t_l\partial t_m}=
    \sum_{m=1}^4\frac{\partial^2 g_m}{\partial t_l\partial t_i}\frac{\partial^2 g_j}{\partial t_k\partial t_m},
    \ \ i,j,k,l=1,2,3,4,  \label{eq:4exWDVV1} \\
    \frac{\partial ^2 g_j}{\partial t_4\partial t_i}=\delta_{ij},\ \ \ \ i,j=1,2,3,4,  \label{eq:4exWDVV2} \\
    Eg_j=\sum_{k=1}^4w_kt_k\frac{\partial g_j}{\partial t_k}=(1+w_j)g_j,\ \ \ j=1,2,3,4  \label{eq:4exWDVV3}
 \end{gather}
  with the constraint $w_1=w_2, w_3=w_4$ and the regularity condition.
  In particular, the coalescence cascade of the Painlev\'e equations
\begin{equation} \label{cd:Painleve}
 \begin{CD}
   \mbox{PVI} @>>> \mbox{PV} @>>> \mbox{PIV}   @.  \\
   @.                             @VVV                     @VVV              @.  \\
   @.                            \mbox{PIII} @>>> \mbox{PII} @>>> \mbox{PI} 
 \end{CD}
\end{equation}
 (except PI) corresponds to
  the degeneration scheme of Jordan normal forms of a square matrix of rank four
 \[
     \begin{CD}
   \begin{pmatrix}  z_{1,0} & {} & {} & {}\\ {} & z_{2,0} & {} & {} \\ {} & {} & z_{3,0} & {} \\ {} & {} & {} & z_{4,0} \end{pmatrix}
   @>>> \begin{pmatrix} z_{1,0} & z_{1,1} & {} & {}\\ {} & z_{1,0} & {} & {} \\ {} & {} & z_{2,0} & {} \\ {} & {} & {} & z_{3,0} \end{pmatrix}
   @>>> \begin{pmatrix} z_{1,0} & z_{1,1} & z_{1,2} & {}\\ {} & z_{1,0} & z_{1,1} & {} \\ {} & {} & z_{1,0} & {} \\ {} & {} & {} & z_{2,0} \end{pmatrix} \\
   @.                             @VVV                     @VVV                \\
   @.       \begin{pmatrix} z_{1,0} & z_{1,1} & {} & {}\\ {} & z_{1,0} & {} & {} \\ {} & {} & z_{2,0} & z_{2,1} \\ {} & {} & {} & z_{2,0} \end{pmatrix}
   @>>> \begin{pmatrix} z_{1,0} & z_{1,1} & z_{1,2} & z_{1,3} \\ {} & z_{1,0} & z_{1,1} & z_{1,2} \\ {} & {} & z_{1,0} & z_{1,1} \\ {} & {} & {} & z_{1,0}
   \end{pmatrix}.
 \end{CD}
 \] 
 \end{theorem}
 
\begin{remark}
  For generic values of $w_i$ (i.e. $w_i\neq w_j$ for $i\neq j$), 
  solutions to the four-dimensional extended WDVV equation 
  (\ref{eq:4exWDVV1})-(\ref{eq:4exWDVV3}) correspond to
  generic solutions to the six-dimensional Painlev\'e equation specified by the spectral type $(21,21,21,21,111)$
  which is studied by T. Suzuki \cite{Su} and its degeneration family.
  In the case of $w_1=w_2$, solutions to (\ref{eq:4exWDVV1})-(\ref{eq:4exWDVV3}) correspond to generic solutions 
  to the Garnier system in two variables and its degeneration family (see e.g.\cite{gausspainleve} for the Garnier system).
  The details will be treated elsewhere.
\end{remark}

\appendix
\section{Isomonodromic deformation of a system of linear differential equations} \label{app:isomonodormicdef}

The aim of this appendix is to prove Lemma~\ref{lem:isomono}.
First we briefly review the theory of isomonodromic deformations of linear differential equations
following \cite{JMU,JM},
and then give a proof of Lemma~\ref{lem:isomono}.

We consider an isomonodromic deformation of an $N\times N$ matrix system of linear differential equations
which has irregular singular points at $x=a_1,\dots,a_n,a_{\infty}=\infty$ on $\mathbb{P}^1$ 
with Poincar\'e rank  $r_{\mu}$ $(\mu=1,\dots,n,\infty)$ respectively:
\begin{equation} \label{lineardiffeq}
 \frac{dY}{dz}=A(z)Y,
\end{equation}
where
\begin{equation*} 
 A(z)=\sum_{\mu=1}^n\sum_{j=0}^{r_{\mu}}A_{\mu,-j}(z-a_{\mu})^{-j-1}-\sum_{j=1}^{r_{\infty}}A_{\infty,-j}z^{j-1},
\end{equation*}
and $A_{\mu,-j}, A_{\infty,-j}$ are $N\times N$ matrices independent of $z$.
We assume that $A_{\mu,-r_{\mu}}$ is diagonalizable as
\begin{equation*}
 A_{\mu,-r_{\mu}}=G^{(\mu)}T^{(\mu)}_{-r_{\mu}}G^{(\mu)-1},\ \ \ \ (\mu=1,\dots,n,\infty)
\end{equation*}
where
\begin{equation*}
 T^{(\mu)}_{-r_{\mu}}=\big(t^{(\mu)}_{-r_{\mu}\alpha}\,\delta_{\alpha\beta}\big)_{\alpha,\beta=1,\dots,N}\ \ \ \ 
 \left\{\begin{array}{l} t^{(\mu)}_{-r_{\mu}\alpha}\neq t^{(\mu)}_{-r_{\mu}\beta}\ \ \mbox{if}\ \alpha\neq\beta, r_{\mu}\geq 1, \\
                         t^{(\mu)}_{0\alpha}\not\equiv t^{(\mu)}_{0\beta}\mod \mathbb{Z}\ \ \mbox{if}\ \alpha\neq\beta, r_{\mu}=0,\end{array}\right.
\end{equation*}
and assume $G^{(\infty)}=1$.
(Generalized Okubo systems do not satisfy these conditions on the eigenvalues of $T^{(\mu)}_{-r_{\mu}}$ in general,
however the arguments in this appendix are valid for them.)
We can take sectors $\mathcal{S}^{(\mu)}_l$ $(l=1,\dots,2r_{\mu})$ centered on $a_{\mu}$,
and there exists a fundamental system of solutions of (\ref{lineardiffeq}) that has the following asymptotic expansion 
on the sector $\mathcal{S}^{(\infty)}_1$ at $z=\infty$:
\begin{equation} \label{fundsol}
 Y(z)\simeq \hat{Y}^{\infty}(z)e^{T^{(\infty)}(z)},
\end{equation}
where $T^{(\infty)}(z)$ is a diagonal matrix
\begin{align*}
 &T^{(\infty)}(z)=\big(e^{(\infty)}_{\alpha}(z)\,\delta_{\alpha\beta}\big)_{\alpha ,\beta =1,\dots,N}, \\
 &e^{(\infty)}_{\alpha}(z)=\sum_{j=1}^{r_{\mu}}t^{(\infty)}_{-j\alpha}\frac{z_{\infty}^{-j}}{-j}+t^{(\infty)}_{0\alpha}\log z_{\infty},\ \ \ 
 z_{\infty}=1/z,
\end{align*}
and $\hat{Y}^{(\infty)}(z)$ is a matrix-valued formal power series of $z_{\infty}$:
\begin{equation}
 \hat{Y}^{(\infty)}(z)=1+Y^{(\infty)}_1z_{\infty}+Y^{(\infty)}_2z_{\infty}^2+\cdots.
\end{equation}
This solution admits the following asymptotic expansions on the other sectors $\mathcal{S}^{(\mu)}_l$:
\begin{equation}
 Y(z)C^{(\mu)-1}S^{(\mu)}_1\cdots S^{(\mu)}_l\simeq G^{(\mu)}\hat{Y}^{(\mu)}(z)e^{T^{(\mu)}(z)},
\end{equation}
where $T^{(\mu)}(z)$ is a diagonal matrix given by
\begin{equation*}
 T^{(\mu)}(z)=\big(e^{(\mu)}_{\alpha}(z)\,\delta_{\alpha\beta}\big)_{\alpha ,\beta =1,\dots,N}
\end{equation*}
with
\begin{equation*}
 e^{(\mu)}_{\alpha}(x)=\sum_{j=1}^{r_{\mu}}t^{(\mu)}_{-j\alpha}\frac{z_{\mu}^{-j}}{-j}+t^{(\mu)}_{0\alpha}\log z_{\mu},
\end{equation*}
\begin{equation*}
 z_{\mu}=\left\{\begin{array}{l} z-a_{\mu},\ \ \mu=1,\dots,n, \\ 
 1/z,\ \ \ \ \mu=\infty, \end{array} \right.
\end{equation*}
and $\hat{Y}^{(\mu)}(z)$ is a matrix-valued formal power series of $z_{\mu}$:
\begin{equation}
 \hat{Y}^{(\mu)}(z)=1+Y^{(\mu)}_1z_{\mu}+Y^{(\mu)}_2z_{\mu}^2+\cdots.
\end{equation}
Here $C^{(\mu)}, S^{(\mu)}_l$ are constant matrices, which are called a {\it connection matrix} and a {\it Stokes multiplier} respectively.

We consider a deformation of (\ref{lineardiffeq}) with $a_{\mu}$ $(\mu=1,\dots,n)$, 
$t^{(\mu)}_{-j\alpha}$ $(\mu=1,\dots,n,\infty; j=1,\dots,r_{\mu}; \alpha=1,\dots,N)$ as deformation parameters
such that $T^{(\mu)}_0,S^{(\mu)}_l,C^{(\mu)}$ are kept invariant. 
(We call such a deformation an isomonodromic deformation.)

The fundamental system of solutions $Y(z)$ to (\ref{lineardiffeq}) characterized by (\ref{fundsol}) is subject to an isomonodromic deformation
with $a_{\mu}, t^{(\mu)}_{-j\alpha}$ as its independent variables
if and only if $Y(z)$ satisfies
\begin{equation} \label{defeq}
 dY(z)=\Omega(z)Y(z),
\end{equation}
where $\Omega(z)$ is a matrix-valued $1$-form
\begin{equation}
 \Omega(z)=\sum_{\mu=1}^{n}B^{(\mu)}(z)da_{\mu}
 +\sum_{\mu=1,\dots,n,\infty}\sum_{j=1}^{r_{\mu}}\sum_{\alpha=1}^{N}B^{(\mu)}_{-j\alpha}(z)dt^{(\mu)}_{-j\alpha},
\end{equation}
 whose coefficients $B^{(\mu)}(z), B^{(\mu)}_{-j\alpha}(z)$ are rational functions with respect to $z$.
From the integrability condition of (\ref{lineardiffeq}) and (\ref{defeq}),
we obtain a system nonlinear differential equations satisfied by $A(z), G^{(\mu)}$:
\begin{align*}
 &dA=\frac{\partial \Omega}{\partial z}+[\Omega, A], \\
 &dG^{(\mu)}=\Theta^{(\mu)}G^{(\mu)},\ \ \ (\mu=1,\dots,n).
\end{align*}
Here we remark that $\Omega, \Theta^{(\mu)}$ are obtained from $A, G^{(\mu)}$ by a rational procedure
which is described by \cite[(3.14) and (3.16)]{JMU}.

Now we step forward to proving Lemma~\ref{lem:isomono}.
Consider an extended generalized Okubo system with the same assumptions as in Section~\ref{sec:okuboiso}:
 \begin{equation} \label{eq:OkubosevA}
         dY=-(zI_N-T)^{-1}(dz+\tilde{\Omega})B_{\infty}Y.
   \end{equation}
The system of differential equations in the $z$-direction of $(\ref{eq:OkubosevA})$ 
    \begin{equation} \label{eq:ordgenOkuboA}
        \frac{dY}{dz}=-(zI_N-T)^{-1}B_{\infty}Y
    \end{equation}
is rewritten to the form of (\ref{lineardiffeq}):
taking an invertible matrix $P$ such that $P^{-1}TP=Z_1\oplus \cdots \oplus Z_n$,
(\ref{eq:ordgenOkuboA}) is written as
\begin{equation} \label{eq:lineardiffeqA}
  \frac{dY}{dz}=\sum_{k=1}^n\sum_{j=0}^{m_{k}-1}A_{k,-j}(z-z_{k,0})^{-j-1}Y,
\end{equation}
where 
\[
   \sum_{j=0}^{m_{k}-1}A_{k,-j}(z-z_{k,0})^{-j-1}=-P\bigl(O\oplus \cdots \oplus (zI_{m_k}-Z_k)^{-1}\oplus \cdots \oplus O\bigr)P^{-1}B_{\infty}.
\]
We remark that $z=\infty$ is a regular singular point with the residue matrix $B_{\infty}$.
Lemma~\ref{lem:isomono} follows from the following lemma:

\begin{lemma}
   The extended generalized Okubo system $(\ref{eq:OkubosevA})$ is equivalent to the completely integrable Pfaffian system consisting of
   $(\ref{eq:lineardiffeqA})$ and $(\ref{defeq})$.
\end{lemma}

\begin{proof}
   The matrix valued $1$-form $\Omega(z)$ in (\ref{defeq}) is uniquely characterized by the following conditions:
   \begin{enumerate}
     \item[(i)] Each entry of $\prod_{k=1}^n(z-z_{k,0})^{m_k}\Omega(z)$ is a polynomial in $z$.
     
     \item[(ii)] $\displaystyle\lim_{z\to \infty}\Omega(z)=0$.
     
     \item[(iii)] The integrability condition is satisfied between (\ref{defeq}) and (\ref{eq:lineardiffeqA}).
   \end{enumerate}
   It is clear that $-(zI_N-T)^{-1}\tilde{\Omega}B_{\infty}$ satisfies the conditions (i)-(iii).
   Hence we have $\Omega (z)=-(zI_N-T)^{-1}\tilde{\Omega}B_{\infty}$.
\end{proof}

\section{Construction of a generalized Okubo system from a system of linear differential equations} \label{app:compOkubo}

In this appendix, we explain following \cite{KawDT,Kaw} how to construct a generalized Okubo system from 
a given linear differential equation of the type of (\ref{lineardiffeq}) with a regular singular point at $\infty$.
This construction is used in Sections~\ref{sec:flatPainleve} and~\ref{sec:coalcas}.

As the first step, we start from a generalized Okubo system
\begin{equation} \label{eq:genOkuboB}
  (zI_N-T)\frac{dY}{dz}=-B_{\infty}Y
\end{equation}
with assumptions (A2) in Section~\ref{sec:okuboiso}
and observe that the rank of (\ref{eq:genOkuboB}) can be reduced by the following procedure.
Let $B_{\infty}=\mbox{diag}(\lambda_1,\dots,\lambda_N)$ and
suppose that $\lambda_{m+1}=\cdots=\lambda_N$, $\lambda_i\neq\lambda_N$ $(1\leq i\leq m)$,
 where $m$ is a natural number less than $N$.
Then in virtue of Remark~\ref{rm:henkan},
we assume $B_{\infty}=\mbox{diag}(\lambda_1,\dots,\lambda_m,0,\dots,0)$ without loss of generality.
Let $G$ be an invertible matrix such that $GTG^{-1}=S$,
where $S$ is the Jordan normal form of $T$:
\[
   S=\begin{pmatrix} z_{1,0} & 1 & {} & O \\ {} & \ddots & \ddots & {} \\ {} & {} & \ddots & 1 \\ O & {} & {} & z_{1,0}\end{pmatrix}
   \oplus \cdots \oplus
   \begin{pmatrix} z_{n,0} & 1 & {} & O \\ {} & \ddots & \ddots & {} \\ {} & {} & \ddots & 1 \\ O & {} & {} & z_{n,0}\end{pmatrix}.
\]
We write the matrix $G$ and its inverse $G^{-1}$ in forms of
\begin{equation}
   G=\begin{pmatrix} C\tilde{R}^{-1} & \tilde{C} \end{pmatrix},\ \ \ G^{-1}=\begin{pmatrix} B \\ \tilde{B} \end{pmatrix}
\end{equation}
respectively, where $\tilde{R}=\mbox{diag}(\lambda_1,\dots, \lambda_m)$
and $B, \tilde{B}, C, \tilde{C}$ are $m\times N, (N-m)\times N, N\times m, N\times (N-m)$ matrices respectively.
Then it holds that 
\[
   -(z-T)^{-1}B_{\infty}=-G^{-1}(z-S)^{-1}GB_{\infty}=-\begin{pmatrix} B(z-S)^{-1}C & O \\ \tilde{B}(z-S)^{-1}C & O \end{pmatrix}.
\]
Hence $\tilde{Y}={}^t(y_1,\dots,y_m)$ satisfies the following $m\times m$ matrix differential equation:
\begin{equation}
  \frac{d\tilde{Y}}{dz}=-B(z-S)^{-1}C\tilde{Y}.
\end{equation}

Next, we give a construction in the opposite direction.
We start from an arbitrary $m\times m$ matrix differential equation:
\begin{equation} \label{eq:lineardiffB}
  \frac{d\tilde{Y}}{dz}=\sum_{k=1}^n\sum_{l=0}^{r_k}\frac{A_k^{(l)}}{(z-a_k)^{l+1}}\tilde{Y},
\end{equation}
where we assume that (\ref{eq:lineardiffB}) has a regular singular point at $z=\infty$
and that $\tilde{R}:=-\sum_{k=1}^nA_k^{(0)}$ is a diagonal matrix: $\tilde{R}=\mbox{diag}(\lambda_1,\dots,\lambda_m)$.
Our goal is to transform (\ref{eq:lineardiffB}) into a generalized Okubo system.
Find a natural number $N$, an $m\times N$-matrix $B$, an $N\times m$-matrix $C$ and an $N\times N$-matrix $S$ in the Jordan normal form
such that 
\[
    \sum_{k=1}^n\sum_{l=0}^{r_k}\frac{A_k^{(l)}}{(z-a_k)^{l+1}}=-B(z-S)^{-1}C,
\]
and then find an $(N-m)\times N$-matrix $\tilde{B}$ and an $N\times (N-m)$-matrix $\tilde{C}$ such that
\[
   \begin{pmatrix} C\tilde{R}^{-1} & \tilde{C} \end{pmatrix}\begin{pmatrix} B \\ \tilde{B} \end{pmatrix}=I_N.
\]
If we obtain such matrices, then 
\begin{equation} \label{eq:genOkubotransB}
  \frac{dY}{dz}=-G^{-1}(z-S)^{-1}GB_{\infty}Y
\end{equation}
is a generalized Okubo system,
where
\[
   G=\begin{pmatrix} C\tilde{R}^{-1} & \tilde{C} \end{pmatrix}, \ \ \ B_{\infty}=\mbox{diag}(\lambda_1,\dots,\lambda_m,0,\dots,0).
\]
In particular, the data consisting of the matrices $\{S, G, B_{\infty}\}$ determines the generalized Okubo system (\ref{eq:genOkubotransB}).
We note that the Jordan normal form $S$ for (\ref{eq:lineardiffB}) is unique provided that the size of $S$ is minimal.
Then $B$ and $C$ is unique up to $(B, C) \sim (Bh^{-1}, hC)$ where $h \in \mathrm{Stab}(S)$ (\cite{Ya}).
This implies the uniqueness of $G$ up to left multiplication by $\mathrm{Stab}(S)$ and right multiplication by $\mathrm{Stab}(B_\infty)$.

\begin{flushleft}
\begin{tabular}{l}
Hiroshi Kawakami\\
Department of Physics and Mathematics, College of Science and Engineering\\
Aoyama Gakuin University \\
Fuchinobe, Sagamihara-shi, Kanagawa 252-5258 JAPAN
\vspace{4mm}
\\
Toshiyuki Mano\\
Department of Mathematical Sciences, Faculty of Science\\
University of the Ryukyus\\
Nishihara-cho, Okinawa 903-0213 JAPAN\\
\end{tabular}
\end{flushleft}


\begin{thebibliography}{99}

\bibitem{AL0} A. Arsie and P. Lorenzoni: From Darboux-Egorov system to bi-flat $F$-manifolds, 
Journal of Geometry and Physics, {\bf 70} (2013), 98-116.

\bibitem{AL1} A. Arsie and P. Lorenzoni: $F$-manifolds, multi-flat structures and Painlev\'e transcendents, 
arXiv:1501.06435

\bibitem{AL2} A. Arsie and P. Lorenzoni: Complex reflection groups, logarithmic connections and bi-flat $F$-manifolds,
Lett. Math. Phys. {\bf 107} (2017), 1919-1961.

 







\bibitem{DH} L. David and C. Hertling: Regular $F$-manifolds: initial conditions and Frobenius metrics,  
 Ann. Sc. Norm. Super. Pisa Cl. Sci. (5) {\bf 17} (2017), 1121-1152.

\bibitem{DR1}
   M. Dettweiler and S. Reiter:
   An algorithm of Katz and its application to the inverse Galois problem.
   J. Symbolic Comput. {\bf 30} (2000), 761-798.
 
\bibitem{DR2}
  M. Dettweiler and S. Reiter:
  Middle convolution of Fuchsian systems and the construction of rigid differential systems.
  J. Algebra {\bf 318} (2007), 1-24.
  

\bibitem{Du} B. Dubrovin: Geometry of 2D topological field theories.
In:Integrable systems and quantum groups. Montecatini, Terme 1993 (M. Francoviglia, S. Greco, eds.) 
Lecture Notes in Math. 1620, Springer-Verlag 1996, 120-348.



\bibitem{HaBook} Y. Haraoka: {\em Linear differential equations on a complex domain.}
Sugakushobou, 2015 (in Japanese).

%

\bibitem{He} C. Hertling: {\em Frobenius Manifolds and Moduli Spaces for Singularities.}
Cambridge University Press, 2002.

\bibitem{HM} C. Hertling and Y. Manin: Weak Frobenius manifolds,  
Int. Math. Res. Not. {\bf 1999}, no. 6 (1999), 277-286.





\bibitem{gausspainleve}
 K. Iwasaki, H. Kimura, S. Shimomura and M. Yoshida:
 {\em From Gauss to Painlev\'{e}.}
 Aspects of Mathematics E16, Vieweg, 1991.

\bibitem{JMU}
 M. Jimbo, T. Miwa and K. Ueno:
 Monodromy preserving deformation of linear ordinary differential equations with rational coefficients. 
 Physica 2D (1981), 306-352.
 
\bibitem{JM}
 M. Jimbo and T. Miwa:
 Monodromy preserving deformation of linear ordinary differential equations with rational coefficients. II,
 Physica 2D (1981), 407-448.
 
\bibitem{KMS1} M. Kato, T. Mano and J. Sekiguchi: Flat structures without potentials, 
 Rev. Roumaine Math. Pures Appl. {\bf 60}, no.4 (2015), 481-505.
 
\bibitem{KMS} M. Kato, T. Mano and J. Sekiguchi: 
Flat structure on the space of isomonodromic deformations. arXiv:1511.01608
 
 \bibitem{KMS3} M. Kato, T. Mano and J. Sekiguchi: Flat structures and algebraic solutions to Painlev\'{e} VI equation. 
  To appear in ``Analytic, Algebraic and Geometric Aspects of Differential Equations" in Trend in Mathematics Series, Springer.
  
\bibitem{KMS4} M. Kato, T. Mano and J. Sekiguchi: Flat Structure and
Potential Vector Fields Related with Algebraic Solutions to Painlev\'e VI
Equation. 
Opuscula Mathematica {\bf 38} (2018), 201-252.

%

\bibitem{Katz} N. M. Katz:
{\em Rigid local systems.} (AM-139),
Princeton University Press, 1996.



\bibitem{KawDT} H. Kawakami:
Generalized Okubo systems and the middle convolution.
Doctor thesis, Univ. of Tokyo, 2009.

\bibitem{Kaw} H. Kawakami:
Generalized Okubo Systems and the Middle Convolution.
Int. Math. Res. Not. {\bf 2010}, no.17 (2010), 3394-3421.






\bibitem{KM} Y. Konishi and S. Minabe: Mixed Frobenius structure and local quantum cohomology,
Publ. Res. Inst. Math. Sci. {\bf 52} (2016), no.1, 43-62.

\bibitem{KoMiSh} Y. Konishi, S. Minabe and Y. Shiraishi: Almost duality for Saito structure and complex reflection groups,
Journal of Integrable Systems {\bf 3} (2018), 1-48.


\bibitem{Lo} P. Lorenzoni: Darboux-Egorov system, bi-flat $F$-manifolds and Painlev\'e VI, Int. Math. Res. Not. {\bf 2014}, no. 12
(2014), 3279-3302. 

\bibitem{LoM} A. Losev and Y. Manin: Extended modular operads, in: Frobenius manifolds, quantum cohomology and singularities
(C. Hertling and M. Marcoli, Eds.), Aspects of Math., E36 (2004), 181-211. 


\bibitem{Ma} Y. Manin: $F$-manifolds with flat structure and Dubrovin's duality,  
Adv. Math. {\bf 198} no. 1, (2005), 5-26.  

\bibitem{Ok} K. Okubo:
Connection problems for systems of linear differential equations. 
Japan-United States Seminar on Ordinary Differential and Functional Equations (Kyoto, 1971),  238-248. 
Lecture Notes in Math., Vol. 243, Springer, Berlin, 1971.



\bibitem{OshiB} T. Oshima:
{\em Fractional calulus of Weyl algebra and Fuchsian differential equations}.
MSJ Memoirs, {\bf 28} 2012.

\bibitem{Oshi1} T. Oshima:
Classification of Fuchsian systems and their connection problem.
Exact WKB analysis and microlocal analysis, RIMS Kokyuroku Besssatsu (2013), B37, 163-192.

\bibitem{Oshi2} T. Oshima:
Katz's middle convolution and Yokoyama's extending operation.
Opuscula Math. {\bf 35} (2015), 665-688.


\bibitem{Ro} S. Romano: $4$-Dimensional Frobenius manifolds and Painlev\'e VI,
Math. Ann. {\bf 360} no. 3, (2014), 715-751.

\bibitem{Sab} C. Sabbah: {\em Isomonodromic Deformations and Frobenius Manifolds. An Introduction}. Universitext. Springer



\bibitem{Sai2} K. Saito: On a linear structure of the quotient variety by a finite reflexion group, 
Preprint RIMS-288 (1979), Publ. RIMS, Kyoto Univ. {\bf 29} (1993), 535-579.


\bibitem{SYS} K. Saito, T. Yano and J. Sekiguchi:
On a certain generator system of the ring of invariants of a finite reflection group.
Comm. Algebra {\bf 8} (1980), 373-408.







 \bibitem{Su} T. Suzuki:
 Six-dimensional Painlev\'e systems and their particular solutions in terms of rigid systems. J. Math. Phys. {\bf 55} (2014), no. 10, 102902, 30 pp. 




 \bibitem{Ya} D. Yamakawa: Middle convolution and Harnad duality, Math. Ann. \textbf{349} (2011), 215--262.
 
%
 \bibitem{Yo} T. Yokoyama:
 Construction of systems of differential equations of Okubo normal form with rigid monodromy.
 Math. Nachr. {\bf 279} (2006), 327-348.
 
\end{thebibliography}
\end{document}